\documentclass[11 pt]{amsart}
\usepackage{graphicx} 
\usepackage{amsmath} 
\usepackage{mathtools}
\usepackage
{amstext, amscd, setspace, tikz-cd,mathtools, enumerate, graphics, latexsym, siunitx}

\usepackage{pst-node}
\usepackage{tikz-cd} 

\usepackage{amssymb}
\usepackage{hyperref}

\usepackage{PTSerif} 

\usepackage[cmtip,all]{xy}
\newcommand{\properideal}{%
\mathrel{\ooalign{$\lneq$\cr\raise.22ex\hbox{$\lhd$}\cr}}}

\newcommand{\properring}{%
\mathrel{\ooalign{$\gneq$\cr\raise.22ex\hbox{$\rhd$}\cr}}}

\newcommand{\Z}{{\mathbb Z}}

\theoremstyle{definition}
\numberwithin{equation}{subsubsection}
\newtheorem{thm}[subsubsection]{{\bf Theorem}}
\newtheorem{cor}[subsubsection]{{\bf Corollary}}
\newtheorem{lem}[subsubsection]{{\bf Lemma}}
\newtheorem{prop}[subsubsection]{{\bf Proposition}}
\newtheorem{defn}[subsubsection]{{\bf Definition}}
\newtheorem{ex}[subsubsection]{{\bf Example}}

\newtheorem{rem}[subsubsection]{{\bf Remark}}

\title[Cluster Magnification, Root Capacity, Unique Chains, Base Change and Ascending Index]{Cluster Magnification, Root Capacity, Unique Chains, Base Change and Ascending Index}
\author{Chandrasheel Bhagwat, Shubham Jaiswal}

\address{Indian Institute of Science Education and Research, Dr.\,Homi Bhabha Road, Pashan, Pune 411008,  INDIA.}
\email{cbhagwat@iiserpune.ac.in, \ jaiswal.shubham@students.iiserpune.ac.in}

\date{\today}

\begin{document}

\begin{abstract}
 This article is inspired from the work of M Krithika and P Vanchinathan in \cite{krithika2023root} 
 and the work of Alexander Perlis in \cite{perlisroots} and \cite{perlis2004roots}. 
 We establish the existence of polynomials for given degree and cluster size over number fields which generalises a result of Perlis. We state the Strong cluster magnification problem and establish an equivalent criterion for that. 
 We also discuss the notion of weak cluster magnification and prove some properties. 
 We provide an important example answering a question about Cluster Towers.
  We introduce the concept of Root capacity and prove some of its properties.
   We also introduce the concept of unique descending and ascending chains for extensions and establish some properties and explicitly compute some interesting examples. We establish results about all these phenomena under a particular type of base change and discuss some other related results about strong cluster magnification and unique chains.
   The article concludes with results about ascending index for a field extension which are analogous to results about cluster size.
    
\end{abstract}

\maketitle


\section{Introduction}
\label{intro}

For an irreducible polynomial over a perfect field, we have the notion of root clusters with combinatorial relation between degree of the polynomial, cluster size and number of clusters. Perlis has proved in \cite{perlisroots} and \cite{perlis2004roots} that the cluster size is independent of choice of root of the irreducible polynomial and that the product of cluster size and number of clusters equals the degree of the polynomial. Perlis has also described the cluster size in terms of the Galois groups associated to the polynomial.
All these notions carry forward to field extensions over the perfect field.\smallskip

Krithika and Vanchinathan proved the Cluster Magnification theorem in \cite{krithika2023root} for irreducible polynomials in which they establish sufficient conditions to magnify the cluster size of an irreducible polynomial. They reformulate the theorem for field extensions as well. \smallskip

In Sec.\ \ref{prelim}, we set up some notations and review the results by Perlis in \cite{perlisroots}, \cite{perlis2004roots} and Krithika-Vanchinathan in \cite{krithika2023root}.\smallskip

In Sec.\ \ref{root cluster size}, we generalise a result of Perlis for number fields that also improves on the generalisation proved previously by Krithika and Vanchinathan in \cite{krithika2023root}. 
The generalisation Thm.\ \ref{n,r}
is as follows.\smallskip

\begin{thm} (Inverse Cluster Size Problem for Number Fields) {\it Let $K$ be a number field. Let $n>2$ and $r|n$. Then there exists an irreducible polynomial over $K$ of degree $n$ with cluster size $r$. }
    
\end{thm}

In the same section, we also present a simple lemma about number of clusters, Lemma \ref{L isom}, which is very useful in giving alternate proofs of results by Perlis and Krithika-Vanchinathan as well as in proving further results.\smallskip

In Sec.\ \ref{Cluster Magnification}, we state the Strong cluster magnification problem and establish the following equivalent criterion for that in Thm.\ \ref{criterion} in terms of Galois groups. For all notations, see Sec.\ \ref{Cluster Magnification}.

\begin{thm} \label{criterion in introduction}
  {\it  An extension $M/K$ is obtained by nontrivial strong cluster magnification from some subextension $L/K$
    if and only if ${{\rm Gal}}(\tilde{M}/K)\cong A\times B$ for nontrivial groups $A$ and $B$ and ${{\rm Gal}}(\tilde{M}/M)\cong A'\times 1$ (under the same isomorphism) for a subgroup $A'\subset A$ with $[A:A']>2$. }
\end{thm}


 We also reformulate the Strong cluster magnification problem for irreducible polynomials. We then state the Weak cluster magnification problem and demonstrate how the notions for strong cluster magnification and weak cluster magnification are actually different.\smallskip

In Sec.\ \ref{cluster towers}, we provide an important example, Example \ref{cluster tower example}, answering a question in \cite{krithika2023root} about Cluster Towers. We also give a group theoretic formulation for cluster towers.\smallskip

In Sec.\ \ref{root capacity}, we introduce the concept of Root capacity. We begin the section with some observations about the group of automorphisms of finite extensions in Prop. \ref{r divide} and Prop. \ref{N Aut} and its corollaries. Then we prove some properties of root capacity in Propositions \ref{elegant}, \ref{rho properties}, \ref{a}, \ref{rho compositum} and \ref{L_M}. We conclude the section with Thm. \ref{hint} 
which is as follows. For notations, see Sec.\ \ref{prelim} and Sec.\ \ref{root capacity}.

\begin{thm} {\it Consider extensions $M/L/K$. If $M\cap \tilde{L}=L$ and $[M:L]=r_K(M)/r_K(L)$, then $M/L$ is Galois. }
    \end{thm}

In Sec.\ \ref{unique chains}, we introduce the concept of unique descending and ascending chains for extensions. Thm.\ \ref{unique descending chain} and Thm.\ \ref{unique ascending chain} encapsulate the important properties of unique chains.  Prop.\ \ref{link unique} links unique ascending and descending chain under certain conditions. We prove an interesting result Prop.\ \ref{symmetric sum} that describes the field $N_1$ in unique descending chain in terms of the sums of elementary symmetric functions.
 In Sec.\ \ref{Interesting Examples}, we compute unique ascending/ descending chains for some interesting examples (see Prop.\ \ref{nth root}, Thm.\ \ref{Perlis unique} and Thm.\ \ref{unique chain Lk}).\smallskip

In Sec.\ \ref{base change}, we define notion of a particular type of base change and then establish 
results Thm.\ \ref{basechange}, \ref{basechangeweak}, \ref{base change root capacity} and \ref{base change unique chain}  about strong/weak cluster magnification, root capacity and unique chains, respectively under this base change. We prove results Thm.\ \ref{strong unique descending} and 
Thm.\ \ref{strong unique ascending} about strong cluster magnification and unique chains.

\smallskip
The article concludes with Sec.\ \ref{Properties of tKL} with some properties of the ascending index
 $t_K(L)$ defined in Thm.\ \ref{unique ascending chain} in the context of unique ascending chain for an extension $L/K$. The ascending index has many properties similar to the cluster size $r_K(L)$ but has no immediate description in terms of roots of the minimal polynomial of $\alpha$ over $K$ when $L = K(\alpha)$. In Sec.\ \ref{Properties of tKL},   
 Prop.\ \ref{basechange-tKL} establishes a base change property for $t_K(L)$. Thm.\ \ref{t strong} establishes an analogue of Cluster Magnification Theorem
 (Thm.\ \ref{Vanchinathan}) for $t_K(L)$ which is as follows. For notations, see Sec.\ \ref{unique chains}.
 
\medskip

 \begin{thm} (Ascending Index Magnification Theorem) {\it Let $M/K$ be obtained by strong cluster magnification with magnification factor $d$. Then 
     $$t_K(M)=d \ t_K(L)\  \text{and}\ u_K(M) = u_K(L).$$}

 \end{thm}

 Finally, Thm.\ \ref{t,r} is an analogue of Thm.\ \ref{n,r} and is as follows.\smallskip

 \begin{thm} (Inverse Ascending Index Problem for Number Fields) {\it Let $K$ be a number field. Let $n>2$ and $t|n$. Then there exists an extension $L/K$ of degree $n$ with ascending index $t_K(L)=t$. }
 \end{thm}

\smallskip

\section{Preliminaries} 
\label{prelim}

\subsection{Root clusters}\hfill

 Let $K$ be a perfect field. We fix an algebraic closure $\bar{K}$ once and for all and work with finite extensions of $K$ contained in $\bar K$.\smallskip

Let $f \in K[t]$ be an irreducible polynomial and let $\alpha$ be a root of $f$ in $\bar{K}$. Since $K$ is perfect, it follows that $f$ has ${\rm deg}(f)$ distinct roots in $\bar{K}$. The cluster of $\alpha$ is defined as the set of roots of $f$ in the field $K(\alpha)$ and its cardinality $r_K(f)$ is called the cluster size of $\alpha$ over $K$.\smallskip

Let $K_f$ be the splitting field of $f$ over $K$ and let $G: = {\rm {{\rm Gal}}}(K_f/K)$. Let $H = {{\rm Gal}}(K_f/K(\alpha))$ be the subgroup of $G$ such that $K(\alpha)$ is the fixed field of $H$.\smallskip

Let $s_K(f)$ be the number of distinct fields of the form $K(\alpha_j)$, with $\alpha_j$ a root of $f$ in $K_f$ for all $1 \leq j \leq {\rm deg}(f)$. 
The following result is proved in \cite{perlisroots} and \cite{perlis2004roots}.\smallskip

\begin{thm}
\label{Perlis}
    (Perlis)
{\it 
\smallskip
     \begin{enumerate}
        \item \it $r_K(f)$ is independent of the choice of $\alpha$.
        
    \smallskip    \item $r_K(f) \ s_K(f) = {\rm deg}(f)$. In particular, $r_K(f) ~| ~{\rm deg}(f)$.
    
 \smallskip       \item $r_K(f)=$ number of roots of $f$ fixed by $H=|{\rm Aut}(K(\alpha)/K)|=[N_G(H): H]$. 
    \end{enumerate}
}
\end{thm}

By the proof of above theorem it follows that $s_K(f)$ is the number of clusters of roots of $f$ in $K_f$.\smallskip

Let $L/K$ be a finite extension of degree $n$ contained in $\bar{K}$ and $\tilde{L}$ be its Galois closure inside $\bar{K}$. Since $K$ is perfect, by primitive element theorem,  $L=K(\alpha)$ with $f$ over $K$ a degree $n$ irreducible polynomial with $\alpha$ as a root in $\bar{K}$. The cluster size of $L/K$ is defined as $r_K(L):=r_K(f)$ which is well defined because of part (3) of  Thm.\ \ref{Perlis} (Corollary 1 in \cite{krithika2023root}). Similarly one can define $s_K(L):=s_K(f)$. Thus we have 
\[ r_K(L) \ s_K(L) =  [L : K]. \]

\begin{rem}
    The cluster size is preserved under isomorphism over $K$. If $M/K$ and $M'/K$ contained in $\bar{K}$ are isomorphic over $K$, then $r_K(M)=r_K(M')$.
\end{rem}

\subsection{Cluster Magnification Theorem} \hfill

We note the following basic fact from Galois theory.

 \begin{rem}
 \label{morandi}
 
     Let $E_1$ and $E_2$ be finite extensions over $K$ and suppose one of them is Galois. Then $E_1$ and $E_2$ are linearly disjoint over $F$ $\iff$ $E_1\cap E_2= K$ (see \cite[Def 20.1 and Example 20.6]{morandi2012field}).
 \end{rem}

The following lemma can be deduced from \cite[Lem.\ 1, Chap.\ 8.15]{jacobsonbasicalgebra2} in combination with Remark\ \ref{morandi}.

\begin{lem}
\label{jacob} 
{\it 
Let $E/K$ be any extension and $F/K$ be Galois extension and let $E'\subset E$. Then $$E\cap F=K\iff E\cap E'F=E'\ \text{and}\  E'\cap F= K.$$
}
\end{lem}

The following result proved in \cite[Sec.\ 3.1]{krithika2023root} is referred to as the Cluster Magnification theorem. The theorem is reformulated in \cite[Sec.\ 4]{krithika2023root}.
\smallskip 

\begin{thm}
\label{Vanchinathan}
(Krithika, Vanchinathan) 
{\it Let $K,f$ and $\alpha$ be as above. Let ${\rm deg}(f)=n > 2$ over $K$ with cluster size $r_K(f)=r$. Assume that there is a Galois extension $F$ of $K$, say of degree $d$, which is linearly disjoint with $K_f$ over $K$. Then there exists an irreducible polynomial $g$ over $K$ of degree $nd$ with cluster size $rd$. ($F$ can be chosen to be $K(\beta)$ for some $\beta$ in $\bar{K}$ so that $K(\alpha,\beta)=K(\alpha + \beta)$ and the minimal polynomial of $\alpha +\beta$ over $K$, has degree $nd$ with cluster size $rd$; $d$ is the magnification factor).\smallskip

    Reformulation : Let $K$ be a perfect field as above and $L/K$ be an extension of degree $n>2$ contained in $\bar{K}$ with cluster size $r_K(L)= r$. Let $F/K$ be any finite Galois extension of degree $d$ contained in $\bar{K}$, which is linearly disjoint with $\tilde{L}$ over $K$. Then the compositum $LF/K$ has degree $nd$ with cluster size $r_K(LF)=rd$ ($d$ is the magnification factor).
   }
\end{thm}
\medskip

\section{Root Cluster Size}\label{root cluster size}

\subsection{Existence of polynomials for given degree and cluster size over number fields}
\hfill
\smallskip

 We present the following theorem which is generalisation of a result in an unpublished note of Perlis \cite[Exercise 4]{perlisroots}, which was for $K=\mathbb{Q}$. This theorem includes even the excluded cases in generalisation done in \cite[Thm. 2]{krithika2023root} namely $n=2r$ where $r$ is odd for $\mathbb{Q}$ and $n=2r$ for any number field $K\neq \mathbb{Q}$.\smallskip

 \begin{thm}
\label{n,r} 
   (Inverse Cluster Size Problem for Number Fields) 
   {\it Let $K$ be a number field. Let $n>2$ and $r|n$. Then there exists an irreducible polynomial over $K$ of degree $n$ with cluster size $r$. }
\end{thm}

\smallskip
Before proving Thm.\ \ref{n,r}, we state some results which we will use in the proof of the theorem.

\medskip

\begin{lem}
\label{sn}
 {\it    The group ${\mathfrak S}_n$ is realisable as Galois group over any number field.}
\end{lem}

\begin{proof}
The reader is referred to V{\"o}lklein \cite{volklein1996groups} for details. The field $\mathbb{Q}$ is hilbertian (see \cite[Def.\ 1.9 and Theorem 1.23]{volklein1996groups}). Furthermore, every finitely generated extension of a hilbertian field is hilbertian (see \cite[Corollary 1.11 ]{volklein1996groups}),
and thus we conclude that every number field  is hilbertian.

\smallskip Let $K$ be a number field. We know that the group ${\mathfrak S}_n$ occurs regularly over every
 field (\cite[Def.\ 1.14 and Example 1.17]{volklein1996groups}). In particular, ${\mathfrak S}_n$ occurs regularly over $K$. 
  The result \cite[Corollary 1.15]{volklein1996groups} says that if a group occurs regularly over a hilbertian field,
   then it is realisable as a Galois group over that field. 
   Hence finally we conclude that ${\mathfrak S}_n$ is realisable as Galois group over $K$.
\end{proof}

\smallskip

The following lemma is the final proposition in Perlis \cite{perlisroots}. We write the proof given by Perlis for the sake of completeness.
\smallskip

\begin{lem}
\label{perm}
{\it    Let $G$ be a transitive subgroup of ${\mathfrak S}_n$ for some $n$. If there exists a finite Galois extension of a field $K$ with Galois group isomorphic to $G$, then there exists an irreducible polynomial $f$ over $K$ of degree $n$
and a labelling of the roots of $f$ so that the Galois group of $f$, viewed as a group permuting roots of $f$, is precisely $G$.}
\end{lem}

\begin{proof}
   Let $G$ act transitively on $n$ symbols $\{1,2,\ldots,n\}$. Let $H\subset G$ be the stabiliser of the symbol $1$. Then there is a canonical labelling of the $n$ cosets in $G/H$ so that $G$ acts on $G/H$ exactly the same way $G$
acts on the original $n$ symbols that is $\{x_1 H, x_2 H,\ldots, x_n H\}$ with $x_j \cdot 1 = j$ and $g \cdot (x_j H)=x_{g \cdot j} H$ for all $g\in G $ and all $1\leq j\leq n$. The action is faithful
and transitive. We have ${\rm Stab}(x_j H)=x_j H x_j^{-1}$. Since the action is faithful, it follows that
$\bigcap \limits_{1\leq j\leq n} (x_j H x_j^{-1}) = \bigcap \limits_{1\leq j\leq n} ({\rm Stab}(x_j H))\ = \{ 1 \}.$
\smallskip

Let $K_G/K$ be finite Galois extension with Galois group isomorphic to $G$. We identify it with $G$. 
Let $L$ be the subfield of $K_G$ fixed by $H$. 
The Galois closure of $L/K$ in $K_G$ is the subfield of $K_G$ corresponding to the intersection of the
conjugates of $H$
in $G$ and that intersection is trivial because of argument in previous paragraph. Hence $K_G$ is the Galois closure of $L/K$. Let $f$ over $K$ be the minimal (hence irreducible) polynomial for a primitive element of $L/K$. We can identify the $n$ roots of $f$ with the $n$
cosets in $G/H$. This is the required polynomial $f$.
\end{proof}
\smallskip

\begin{lem}
    \label{conradbeautiful}
{\it Let $K$ be a perfect field and $K'/K$ be a finite extension. If for a group $G$, direct products $G^n$ are realisable as Galois group over $K$ for each $n\in \mathbb{N}$ then they are realisable as Galois groups over $K'$.   }
\end{lem}

\begin{proof}
    We will mimic the proof of Thm.\ 4.2 in \cite{conrad2023galois} which states that if every finite group can be realised as a Galois group over $\mathbb{Q}$ then every finite group can be realised as a Galois group over any finite extension of $\mathbb{Q}$.\smallskip
    
    Since $K$ is perfect and $K'/K$ is finite, we have that $K'/K$ is separable. Hence, $K'/K$ has finitely many intermediate fields. Let $n$ be the number of these intermediate fields (including $K'$ and $K$). Now $G^n$ is realisable over $K$ by assumption, say for $E/K$ Galois, we have ${\rm Gal}(E/K)\cong G^n$.\\ We have normal
subgroups $N_i = G \times G \times \dots \times {1} \times \dots \times G$ of $G^n$ for $1 \leq i \leq n$ where the $i$th coordinate is trivial and there is no restriction in other
coordinates. So $N_i \cong G^{n-1}$. Let $E_i$ be the subfield of $E$ corresponding to
$N_i$, so $E_i/K$ is Galois with ${\rm Gal}(E_i/K) \cong G^n / N_i \cong G$.\smallskip

Now for $i\neq j$, $E_i\cap E_j$ corresponds to subgroup generated by $N_i$ and $N_j$ which is $G^n$. Hence $E_i\cap E_j=K$. Suppose for some $i\neq j$ we have $E_i\cap K'=E_j\cap K'$. Since $E_i\cap E_j=K$, we get $E_i\cap K'=E_j\cap K'=K$. Now suppose that all $E_i\cap K'$ are distinct. Since we have $n$ intermediate fields of $K'/K$, $E_i\cap K'=K$ for some $i$. In either case we get an $i$ such that $E_i\cap K'=K$. Hence ${\rm Gal}(E_iK' / K')\cong {\rm Gal} (E_i/K)\cong G$. This realises $G$ as a Galois group over $K'$. By replacing $G$ with $G^m$ for any $m\in\mathbb{N}$ in the above argument, we can realise $G^m$ over $K'$ for any $m$. \end{proof}\smallskip

\begin{rem}
 In the above proof, we observe that for any $n$ we have some $E/K$ Galois with Galois group $G^n$ and $N_i$ normal subgroups of $G^n$ and $E_i$ subfield of $E$ corresponding to $N_i$. We observe that $N_i$ are not conjugate to each other in $G^n$ and they pairwise generate $G^n$. Hence $E_i$ are not isomorphic to each other over $K$ and are pairwise linearly disjoint over $K$ with $G$ as Galois group of each $E_i/K$. Thus we conclude the following.
 \smallskip
 
 Let $K$ be a perfect field and $K'/K$ be a finite extension and suppose for a group $G$, direct products $G^n$ are realisable as Galois group over $K$ for each $n\in \mathbb{N}$. Then we get arbitrarily large finite families of 
Galois extensions of $K$  inside a fixed $\bar{K}$ which are pairwise non-isomorphic over $K$  and are pairwise linearly disjoint over $K$ with each having Galois group $G$ over $K$. This statement also holds for $G^m$ in place of $G$ for any $m\in \mathbb{N}$ as well as for $K'$ in place of $K$.

\end{rem}

\smallskip

Now we prove Thm.\ \ref{n,r}.

\begin{proof} 
   
   Suppose $r=1$. By Lemmas \ref{sn} and \ref{perm}, there exists an irreducible polynomial $f$ over $K$ of degree $n$ with Galois group ${\mathfrak S}_n$. This $f$ satisfies $r_K(f)=1$.    \smallskip

   Now suppose $r>1$. In solutions of Exercises 3 and 4 in \cite{perlisroots}, a solvable group $G\subseteq{\mathfrak S}_n$ is constructed with the properties that its action is transitive on $n$ points, and a point stabiliser fixes precisely $r$ points. The construction is as follows: We divide the $n$ points into $n/r=s$ packets of size $r$. Let $G$ be the group of permutations on these $n$ points generated by independent
cyclic permutations on each packet, together with a cyclic permutation
on the overall set of packets. Hence $G$ is transitive. This construction of $G$ has the explicit description of a semidirect product of an $s$-fold direct product of cyclic groups $\Z/r\Z$ and a cyclic group $\Z / s\Z$. A semidirect product group law on $G$ is given by \[ ((a_1,\dots, a_s),b)\cdot((c_1,\dots,c_s),d)=((a_1,\dots, a_s) + (b\cdot (c_1,\dots, c_s)), b+d), \]
  where $b\cdot(c_1,\dots, c_s)=(c_{b+1},\dots, c_s,c_1,\dots, c_{b})$. \smallskip

Thus, $G=(\Z/r\Z)^s \rtimes \Z / s\Z,$
\text{where each}\ $\Z/r\Z$ \text{permutes points in a packet and action of}\ $\Z / s\Z$ \text{permutes}\ 
$s\ \text{copies of}\ \Z/r\Z.$\smallskip

It is easy to see that any point stabiliser is isomorphic to $(\Z/r\Z)^{s-1}$. The group $G$ is solvable since the following chain has successive cyclic quotients.
$$1\subseteq \Z/r\Z\subseteq(\Z/r\Z)^2\subseteq \dots \subseteq (\Z/r\Z)^s\subseteq G.$$  
    
    Since direct product of solvable groups is solvable, direct products $G^i$ for $i\in \mathbb{N}$ are solvable. By Shafarevich's theorem 
    (\cite{Shafarevich1989}), $G^i$ for $i\in \mathbb{N}$ are realisable as Galois groups over $\mathbb{Q}$. Hence by Lemma \ref{conradbeautiful}, $G$ is realisable as Galois group over number field $K$. By Lemma \ref{perm}, there exists an irreducible polynomial $f$ over $K$ of degree $n$
and a labelling of the roots so that the Galois group of $f$, viewed as a group permuting roots of $f$, is precisely $G$. This $f$ satisfies $r_K(f)=r$.
 \end{proof}
 
\smallskip

\subsection{A simple lemma about $s_K(L)$} 
\label{alternate proofs}
\hfill \smallskip

We begin this subsection by giving an alternate proof for last equality in (3) of Thm.\ \ref{Perlis}, which is stated in \cite{perlis2004roots} and proved in first Proposition in unpublished note of Perlis \cite{perlisroots}. The equality states: $r_K(f)=[N_G(H):H].$

\begin{proof} 
   We observe that a field is isomorphic to $K(\alpha)$ over $K$ if and only if it is of the form $K(\alpha')$ for some root $\alpha'$ of $f$. All these fields are contained in $K_f$.  By Galois correspondence, $K$-isomorphic subfields of a Galois extension over $K$ correspond to conjugate subgroups of its Galois group. Hence, $s_K(f)=$ number of distinct $K(\alpha')=$ number of distinct subgroups of $G$ that are conjugate to $H$ in $G = [G:N_G(H)]$. By Thm.\ \ref{Perlis} (2), we are done.
\end{proof}

\medskip

We state the simple observation used above, as a lemma.

\begin{lem}
\label{L isom}
{\it   Let $K$ be perfect field. For finite $L/K$, $s_K(L)$ (as defined in Sec. \ref{prelim}) is the number of distinct fields inside $\bar{K}$ isomorphic to $L$ over $K$. }
\end{lem}

\begin{proof}
   By primitive element theorem, $L=K(\alpha)$ with $\alpha$ root of some irreducible polynomial $f$ over $K$. Now, $L'$ is isomorphic with $L$ over $K$ $\iff$ $L'=K(\alpha')$ for some root $\alpha'$ of $f$. Thus $s_K(L)=s_K(f)=$ number of distinct $K(\alpha')$ for $\alpha'$ root of $f$= number of distinct fields isomorphic to $L$ over $K$. 
\end{proof}

\begin{rem}
\label{tilde L}
Let $L_1,L_2,\dots,L_{s_K(L)}$ be the distinct fields as in above Lemma\ \ref{L isom}. Hence we have $\tilde{L}=L_1 L_2 \dots L_{s_K(L)}$, that is the Galois closure of $L/K$ is compositum of distinct fields isomorphic to $L$ over $K$.

\end{rem}
\smallskip

Using Lem.~\ref{L isom}, we will give an alternate proof for Cluster Magnification theorem Thm.\ \ref{Vanchinathan}, Thm.\ 1 in \cite{krithika2023root} in Sec.\ \ref{base change}.\smallskip

\begin{rem}
Let $K$ be a number field. Note that $^nP_k$ and $^nC_k$ are integers with $^nP_k=k!$ $ ^nC_k$. By Thm. \ref{n,r}, 
we get irreducible polynomial over the field with degree $^nP_k$ and cluster size $k!$. The following theorem says that
 this is also true under some condition for a general perfect field.
\end{rem}

By using the above Lemma ~\ref{L isom}, we give an alternate proof for Thm. 3 in \cite{krithika2023root}.

\begin{thm}
\label{nPk}
{\it    Let $K$ be a perfect field. Let $f$ over $K$ be irreducible of degree $n$ with Galois group ${\mathfrak S}_n$. 
For $1 \leq k \leq n-2$, let $L_k$ be
an extension of $K$ obtained by adjoining any $k$ roots of $f$ in $\bar{K}$. Let $g$ be the minimal polynomial over $K$ for a primitive element of $L_k$. This polynomial has degree $^nP_k$ and cluster size $k!$.
}
\end{thm}

\begin{proof}
    Let $\alpha_1,\alpha_2,\ldots,\alpha_n$ be roots of $f$ in $\bar{K}$. Let $L_k=K(\alpha_1,\alpha_2, \ldots, \alpha_k)$. We have that degree of $L_k/K$ is $n(n-1)\ldots(n-k+1)= $$ ^nP_k$. Since Galois group of $f$ is ${\mathfrak S}_n$, we have $L'$ is isomorphic to $L_k$ over $K$ $\iff$ $L'=K(\alpha_{i_1},\alpha_{i_2},\ldots,\alpha_{i_k})$ for $k$ roots $\alpha_{i_j}: 1 \leq j \leq k$ of $f$. By Lemma\ \ref{L isom}, $s_K(L_k)$ is number of distinct fields inside $\bar{K}$ isomorphic to $L_k$ over $K$ which is precisely the number of ways of choosing $k$ roots from $n$ roots which is $^nC_k$. By Thm.\ \ref{Perlis} (2), $r_K(L_k)=k!$.
\end{proof}
\medskip


\section{Cluster Magnification}\label{Cluster Magnification}

\subsection{Strong cluster magnification} \label{sec-SCM}\hfill

Let $M/K$ be a finite extension of degree $m$ with $r_K(M)=k$. 
\begin{defn}
\label{SCM}
\hfill \smallskip

\noindent
$M/K$ is said to be obtained by strong cluster magnification from a subextension $L/K$ if we have the following:
\smallskip
 
 \begin{enumerate}
 \item  $[L:K] = n > 2,$ 
 \smallskip
      
\item there exists a finite Galois extension $F/K$ such that the Galois closure $\tilde{L}$ of $L$ in $\bar{K}$ and $F$ are linearly disjoint over $K$. \smallskip

\item $LF=M$.\smallskip

 \end{enumerate}

The number $[F:K]$ is called the magnification factor and denoted by $d$. 
The magnification is called trivial if $F = K$ and nontrivial otherwise. 
\end{defn}

\begin{rem}
   Suppose we have an extension $L/K$,and a Galois extension $F/K$ such that $\tilde{L}\cap F=L\cap F$. Then $LF/ (L\cap F)$ is obtained by strong cluster magnification from $L/ (L\cap F)$.
\end{rem}

\begin{rem}
    Let $LF/K$ be obtained by strong cluster magnification from $L/K$ through $F/K$. If $K\subset L' \subset L$. Then $L'F/K$ is obtained by strong cluster magnification from $L'/K$ through $F/K$.
\end{rem}

We prove the following hereditary property for strong cluster magnification. 

\smallskip
\begin{prop} \hfill \medskip

\noindent
{\it Let $M/K$ be obtained by strong cluster magnification from $L/K$ through $F/K$ as in Def \ref{SCM}. Then for any $K\subset K'\subset L$ the extension $M/K'$ is obtained by strong cluster magnification from $L/K'$ through $K'F/K'$ with same magnification factor.}
\end{prop}

\begin{proof} We check that the conditions in Def.\ \ref{SCM} hold.\smallskip

    Let $L_1$ be Galois closure of $L/K'$. So $L_1\subset \tilde{L}$. Since $\tilde{L}$ and $F$ are linearly disjoint over $K$, we conclude that $L_1$ and $F$ are linearly disjoint over $K$. Hence by Lemma \ref{jacob} we have $$L_1\cap F= K \iff L_1 \cap K'F=K' \ \text{and}\ K'\cap F=K.$$

    Hence $K'F/K'$ is Galois and $L_1$ and $K'F$ are linearly disjoint over $K'$. Also $M=LF=LK'F$ and hence we are done. The magnification factor is same since $[F:K]=[K'F:K']$.
\end{proof}

\smallskip

Let $\tilde{M}$ be Galois closure of $M/K$ inside $\bar{K}$. Let $G'= {{\rm Gal}}(\tilde{M}/K)$. Let $H' = {{\rm Gal}}(\tilde{M}/M)$ be the subgroup of $G'$ with fixed field $M$. Hence $H'$ is normal in $G'$ if and only if $H'$ is trivial. \smallskip


\begin{prop} \label{prop} \hfill \medskip

\noindent   {\it Suppose $M/K$ is obtained by strong cluster magnification from  $L/K$. Let $\tilde L$ and $F/K$ be as in the Def.\ \ref{SCM}  and let $R: = {{\rm Gal}}(F/K)$. Let $G = {{\rm Gal}}(\tilde{L}/K)$ and $H= {{\rm Gal}}(\tilde{L}/L)$. Then the following hold.

\smallskip 

\begin{enumerate}

\smallskip \item $r_K(M)=r_K(L)\ [F:K], \quad s_K (M) = s_K(L)=[G:N_G(H)]$.

\smallskip  \item $\tilde{L} F = \tilde{M}$.

\smallskip \item $G'\cong G\times R$ where isomorphism is given by  $\lambda\in G' \mapsto (\lambda|_{\tilde{L}},\lambda|_{F})$.

\smallskip \item  Furthermore $H'\cong H\times \{e\} \subset G\times R$ under the above isomorphism.

\smallskip \item $F$ is uniquely determined by $L$ and $M$.

    \end{enumerate}
}
\end{prop}

\begin{proof} 
\hfill

\begin{enumerate}

\item From (2) and (3) in Thm.\ \ref{Perlis}, we have
$s_K(L)=[G:N_G(H)]$. Now from Thm.\ \ref{Vanchinathan}, 
$[M:K]=[L:K]\ [F:K]$ and $r_K(M)=r_K(L)\ [F:K]$. \smallskip

 Also from Thm.\ \ref{Perlis}, $[M:K]= r_K(M)\ s_K(M)$ and $[L:K]= 
 r_K(L)\ s_K(L)$. Hence $s_K (M) =s_K(L)$.\smallskip

\item Since $M=LF$, we have $\tilde{L} F\subset \tilde{M}$. Since $\tilde{L}/K$ and $F/K$ are Galois it follows that $\tilde{L} F/K$ is Galois. Thus, $\tilde{L} F = \tilde{M}$.\smallskip

\item $\tilde{L}$ and $F$ are linearly disjoint over $K$. Since $F/K$ is Galois, it follows that $\tilde{L}\cap F=K$. Therefore, by (2) and  \cite[Thm.\ 2.1]{conrad2023galois}, we conclude that $G'\cong G\times R$ under the given isomorphism.\smallskip

\item Let $\lambda\in G'$. We have $$\lambda\in H'\iff \lambda |_{M}=id_M \iff \lambda |_L=id_L \ \text{and}\  \lambda |_{F}=id_F \iff \lambda |_{\tilde{L}}\in H\subset G\ \text{and}\ \lambda |_F = 1 \in R.$$

\item     We have isomorphism $G'\cong G \times R$. Hence $G'=G_0 R_0$ where $G_0,R_0\subset G'$ with $G_0\cong G\times 1$ and $R_0\cong 1\times R $ under above isomorphism. Furthermore, $\tilde{L}=\tilde{M}^{R_0}$ and $F=\tilde{M}^{G_0}$. Now if a subextension $L$ of $M$ is given, then $\tilde{L}$ is uniquely determined inside $\tilde{M}$. Thus $R_0$ is uniquely determined inside $G'$ which implies that $G_0$ is uniquely determined inside $G'$. Hence $F$ is uniquely determined from $L$ and $M$. \end{enumerate} 

\vspace{-0.7 cm}
\end{proof}

\smallskip

In view of property (1) in Prop.\ \ref{prop}, we see that ${r_K(L)} | {r_K(M)}$ and the ratio ${r_K(M)}/{r_K(L)}$ is indeed same as the degree $d$ = $[F: K]$, which is the magnification factor for $M / L$ over $K$ as defined earlier.
\medskip

\noindent {\bf A criterion for strong cluster magnification:} We now establish an equivalent criterion for strong cluster magnification for a field extension in terms of Galois groups.\smallskip

\begin{thm} \label{criterion}
{\it     An extension $M/K$ is obtained by nontrivial strong cluster magnification from some subextension $L/K$
    if and only if ${{\rm Gal}}(\tilde{M}/K)\cong A\times B$ for nontrivial groups $A$ and $B$ and ${{\rm Gal}}(\tilde{M}/M)\cong A'\times 1$ (under the same isomorphism) for a subgroup $A'\subset A$ with $[A:A']>2$. }
\end{thm}

\begin{proof}
    Suppose $M/K$ is obtained by nontrivial strong cluster magnification from a subextension $L/K$. From Prop.\ \ref{prop} (3) and (4), we get $A=G, B=R$ and $A'=H$ with the required conditions since, $d=|R|>1,n=[G:H]>2$. \smallskip

    Conversely, suppose $G'\cong A\times B$ for nontrivial groups $A$ and $B$ and $H'\cong A'\times \{e\}$ (under the same isomorphism) for a subgroup $A'\subset A$ with $[A:A']>2$. We identify $G'$ and $H'$ with their images under the isomorphism. Now we check the three conditions of Def.\ \ref{SCM} for $M/K$.\smallskip
    
\begin{enumerate}
    
     \item Since $1\times B$ is normal in $G'$, we conclude that $\tilde{M}_B :=\tilde{M}^{1\times B}$ is Galois over $K$ with Galois group $A$. Let $L :=\tilde{M}^{A'\times B}$. Hence $L/K$ has degree $n=[A:A']>2$.\smallskip
     
         \item Since $A\times 1$ is normal in $G'$, we conclude that $F :=\tilde{M}^{A\times 1}$ is Galois over $K$ with Galois group $B$ and degree $d=|B|$. Let $\tilde{L}$ be Galois closure of $L$ in $\bar{K}$. Since, $L\subset \tilde{M}_B$, we have $\tilde{L}\subset \tilde{M}_B$. The intersection of fields $\tilde{M}_B\cap F$ corresponds to the subgroup generated by $A\times 1$ and $1\times B$ which is $G'$. Hence $\tilde{M}_B\cap F=K$. Thus $\tilde{L}\cap F=K$. So $\tilde{L}$ and $F$ are linearly disjoint over $K$.\smallskip
         
 \item Now, $M=\tilde{M}^{A'\times 1}$. Hence $L,F\subset M$, thus $LF\subset M$. Since $\tilde{L}\cap F=K$, we conclude $L\cap F = K$. Hence $[LF:K]=[L:K]\ [F:K] = nd$. Also, $[M:K] = [G':H'] = [A:A']\ |B|=nd$. Hence $LF=M$.

 \end{enumerate}
The magnification is nontrivial since $B$ is nontrivial group.
\end{proof}

\begin{rem}
    In the above proof of the converse part, we can additionally conclude $\tilde{M}_B=\tilde{L}$. Since, $ \tilde{M}_B \cap F=\tilde{L}\cap F=K$ we get $[F:K]=[\tilde{M}_B F: \tilde{M}_B]=[\tilde{L}F:F]$. Now, $ \tilde{M}_B F\subset \tilde{M}$ corresponds to intersection of the groups $A\times 1$ and $1\times B$ which is trivial. Hence $\tilde{M}_B F=\tilde{M}$. From prop.\ \ref{prop} (2), $\tilde{M}=\tilde{L}F$. Hence, $[\tilde{M}_B :K]=[\tilde{L}:K]$. Thus, $ \tilde{M}_B =\tilde{L}$.
\end{rem}

\begin{cor}
    \label{galois}
    \hfill \medskip

{\it \noindent
    Let $M/K$ be Galois. Then $M/K$ is obtained by nontrivial strong cluster magnification from some subextension $L/K$
    if and only if ${{\rm Gal}}(M/K)\cong A\times B$ for nontrivial groups $A$ and $B$ with $|A|>2$. If this happens then, $L/K$ is also Galois. }
\end{cor}

\begin{proof}
    Since $M/K$ is Galois, $\tilde{M}=M$, $G'={\rm Gal}(M/K)$ and $H'$ is trivial. So $A'$ is trivial and $[A:A'] =|A|$. Also $\tilde{L} = \tilde{M}_B =  L$. Hence, $L/K$ is Galois.
\end{proof}

\smallskip
The following says, strong cluster magnification behaves well with respect to $K$-isomorphisms.

\smallskip
\begin{rem}
\label{isom}
    Let $M'/K$ be contained in $\bar{K}$ and $\sigma : M \rightarrow M'$ be an isomorphism over $K$. If $M/K$ is obtained by strong cluster magnification from $L/K$, then $M'/K$ is obtained by strong cluster magnification from $\sigma(L)/K$. 
\end{rem}



    

\begin{rem}

From the proof of Thm.\ \ref{criterion}, we get the following way to construct all fields $M/K$ which are obtained by nontrivial strong cluster magnification from some subextension $L/K$.

\smallskip

    Suppose $A\times B$ is realisable as a Galois group over $K$ for nontrivial groups $A$ and $B$ with a subgroup $A'\subset A$ with $[A:A']>2$ such that $\bigcap \limits_{a\in A}\ a A' a^{-1}= 1$. Let $P$ be such that ${\rm Gal}(P/K)=A\times B$. Then our required fields are $M=P^{A'\times 1}$ and $L=P^{A'\times B}$.

\end{rem}

\medskip

\subsection{Strong cluster magnification problem for irreducible polynomials} 
\hfill 

Let $g$ be an irreducible polynomial over $K$ with degree  $m$ and $r_K(g) = k$. 

\begin{defn}
\hfill \smallskip

\noindent
We have the following equivalent definitions:
\begin{enumerate}
    \item The polynomial $g$ is said to be obtained by strong cluster magnification from a polynomial $f$ over $K$ if we have the following:\smallskip  
     
     \begin{enumerate}
         \item an extension $K(\alpha)/K$ of degree $n>2$ with $f$ as the minimal polynomial of $\alpha$
          over $K$ and $r_K(f)=r$,\smallskip  
          
         \item there exists a Galois extension $F/K$ of degree $d$ such that $K_f$ and $F$ are linearly disjoint over $K$, and \smallskip  
         
 \item $K(\alpha)F=K(\gamma)$ where $\gamma$ is some root of $g$ in $\bar{K}$.\smallskip  

 \end{enumerate}

The magnification is called trivial if $d=1$ and nontrivial otherwise. ($d$ is the magnification factor).\smallskip  

\item  The polynomial $g$ over $K$ is said to be obtained by strong cluster magnification from a polynomial $f$ over $K$, if for some root $\gamma$ of $g$ in $\bar{K}$, the field extension $M=K(\gamma)$ over $K$ is obtained by strong cluster magnification from $L/K$ with $L=K(\alpha)$, where $\alpha$ is a root of the irreducible polynomial $f$. 
\smallskip  
\end{enumerate}
     
\end{defn}

\begin{rem} 
 Let $s_K(g)=s$. Let $\{ \gamma_1, \gamma_2,\ldots, \gamma_s\}$ be a complete set of representatives of the clusters of roots of $g$ in $\bar{K}$. Let $M_i=K(\gamma_i)$ for $1\leq i\leq s$.\smallskip  
 
\begin{enumerate}
     
\item  All $M_i$'s are mutually isomorphic by mapping $\gamma_i$'s to each other. For every $i$, $M_i/K$ is an extension of degree $m$ contained in $\bar{K}$ with $r_K(M_i)=r_K(g)=k$.\smallskip  

\item  By Remark\ \ref{isom}, if for some $i$, $M_i/K$ is obtained by strong cluster magnification from $L_i/K$ then for each $1\leq j\leq s$, $M_j/K$ is obtained by strong cluster magnification from some subextension  $L_j/K$. \smallskip  

\item  More precisely, if $L_i=K(\alpha_i)$ for a root $\alpha_i$ of the irreducible polynomial $f$ then, by isomorphism of $M_i$ and $M_j$, we get $L_j=K(\alpha_j)$ where $\alpha_j$ is a root of $f$.\smallskip  

\end{enumerate}   
 
\end{rem}
    
Because of the above remark, if strong cluster magnification holds for some root of $g$, then it holds for every root of $g$. Hence we can work with any root $\gamma$ of $g$ in $\bar{K}$. Let $M=K(\gamma)$. Let $K_g$ be splitting field of $g$ over $K$ inside $\bar{K}$. Let $G'={{\rm Gal}}(K_g/K)$. Let $H'\subset G'$ be subgroup with $K(\gamma)$ as the fixed field that is $H'={{\rm Gal}}(K_g/K(\gamma))$.\medskip

By Thm.\ \ref{criterion}, we get an equivalent criterion for strong cluster magnification of an irreducible polynomial.

\begin{thm}
\label{scm poly}
{\it A polynomial $g$ over $K$ is obtained by nontrivial cluster magnification from an $f$ over $K$
    if and only if ${{\rm Gal}}(K_g/K)\cong A\times B$ for nontrivial groups $A$ and $B$ and (for $\gamma$ as above) ${{\rm Gal}}(K_g/K(\gamma))\cong A'\times 1$ (under the same isomorphism) for a subgroup $A'\subset A$ with $[A:A']>2$ and $f$ is the minimal polynomial for a primitive element of $(K_g)^{A'\times B}$ over $K$. 
}     
\end{thm}

\begin{defn}\hfill \smallskip

\noindent
    An irreducible polynomial $g$ over $K$ is called polynomial with primitive clusters if it is not obtained by a nontrivial strong cluster magnification over $K$. (this notion occurs in \cite{krithika2023root} as well).
\end{defn}

\begin{ex}
Some simple examples / cases of polynomials with primitive clusters $g$ over $K$ follow from Thm.\ \ref{Vanchinathan} and Thm.\ \ref{scm poly}.
\begin{enumerate}

\smallskip \item ${\rm deg}\ g=4$  or a prime $p>2$.

\smallskip    \item $|{{\rm Gal}}(K_g/K)|=4$. 

\smallskip \item  ${{\rm Gal}}(K_g/K)$ is not a direct product of two nontrivial groups. In particular, the case when ${{\rm Gal}}(K_g/K)$ is simple. In particular, the case $|{{\rm Gal}}(K_g/K)|$ is a prime $p >2$.

    
\end{enumerate}
    
\end{ex}
\smallskip

\subsection{Weak cluster magnification}\hfill

\begin{defn}\hfill \smallskip

\noindent
    $M/K$ is said to be obtained by weak cluster magnification from a subextension $L/K$ if $r_K(L)|r_K(M)$. We call $d=r_K(M)/r_K(L)$ as the magnification factor. The magnification is called trivial if $d=1$ and nontrivial otherwise.
\end{defn}
\smallskip

\begin{rem}
    From Def \ref{SCM} and Thm.\ \ref{Vanchinathan}, if $M/K$ is obtained by strong cluster magnification from $L/K$ then $M/K$ is obtained by weak cluster magnification from $L/K$. This justifies the word `weak' in above definition.
\end{rem}
\smallskip

\begin{ex}
  Consider $M/L/K$ where $M/K$ is Galois and $L/K$ is not Galois (a particular case is $M=\tilde{L}$ for $L/K$ not Galois). So $r_K(L)\neq [L:K]$ and $r_K(M)=[M:K]$. Now $r_K(L)~|~[L:K]$ and $[L:K]~|~[M:K]$. Hence $r_K(L)~|~r_K(M)$ and $r_K(L)\neq r_K(M)$. Thus $M/K$ is obtained by nontrivial weak cluster magnification from $L/K$. \smallskip

   We claim that $M/K$ is not obtained by strong cluster magnification from $L/K$. Assume the contrary. Then we must have by Def.\ \ref{SCM} and Lem.\ \ref{jacob} that $M\cap \tilde{L}=L$ which is a contradiction since $L/K$ is not Galois. Moreover $M/K$ is not obtained by strong cluster magnification of $L_1/K$ for any $\tilde{L} \supsetneq L_1 \supset L$.

\end{ex}
\smallskip

\begin{ex}\label{j=n-k}

 Let $f$ over $K$ be irreducible of deg $n$ with Galois group ${\mathfrak S}_n$ with roots $\alpha_i\in \bar{K}$ for $1\leq i\leq n$. For $1 \leq k \leq n-2$, let $L_k=K(\alpha_1,\dots, \alpha_k)$. Then for $j<k$, $L_k/K$ is obtained by nontrivial weak cluster magnification from $L_j/K$ but is not obtained by nontrivial strong cluster magnification from $L_j/K$ since $\tilde{L_{j}}\cap L_k=L_k\neq L_j$.\smallskip

 One can also verify that $[L_k:L_j]=r_K(L_k)/r_K(L_j)\iff k > n/2 \ \text{and}\ j=n-k$.
    
\end{ex}

\medskip

\begin{ex}
    Let $K=\mathbb{Q}$ and $\xi_n$ be $n$-th primitive root of unity. 
    
    \begin{enumerate}
        \item  Let $M=\mathbb{Q}(\xi_{2^k})$ and $L=\mathbb{Q}(\xi_{2^{k-1}})$ for $k\geq 4$. Now $M/K$, $L/K$ and $M/L$ are Galois. Hence $M/K$ is obtained by nontrivial weak cluster magnification from $L/K$ with magnification factor $[M:L]=2$. Also ${\rm Gal}(M/K)\cong \mathbb{Z}/2 \mathbb{Z}\times \mathbb{Z}/(2^{k-2}) \mathbb{Z}$ and ${\rm Gal}(L/K)\cong \mathbb{Z}/2 \mathbb{Z}\times \mathbb{Z}/(2^{k-3}) \mathbb{Z}$. By uniqueness in structure theorem for finite abelian groups ${\rm Gal}(M/K)\not \cong \mathbb{Z}/2\mathbb{Z} \times {\rm Gal}(L/K)$. Hence by Corollary \ref{galois}, we conclude that $M/K$ is not obtained by nontrivial strong cluster magnification from $L/K$.\smallskip

        We can use a similar argument as above to conclude that
      for integers $k>j\geq 3$ for prime $p=2$ and integers $k>j\geq 2$ for prime $p\neq 2$, $M=\mathbb{Q}(\xi_{p^k})$ is not obtained by nontrivial strong cluster magnification from $L=\mathbb{Q}(\xi_{p^j})$.\smallskip

        \item Let $n, l$ be integers such that $6 < l < n$, $l | n$ such that $n = l m$ and $\text{gcd}(l,m)=1$.
     Let $M=\mathbb{Q}(\xi_{n})$ and $L=\mathbb{Q}(\xi_l)$. Since gcd$(l,m)=1$, we have $v_p(m)=v_p(n)$ for $p|m$. We have \smallskip
\[
\begin{split} {\rm Gal}(M/K) & \cong 
\prod \limits_{p|n}\ (\mathbb{Z}/p^{v_p(n)}\mathbb{Z})^{\times} =
 \prod \limits_{p | l}\ (\mathbb{Z}/p^{v_p(n)} \mathbb{Z})^{\times}\  \times \prod \limits_{p | m}\ (\mathbb{Z}/p^{v_p(n)} \mathbb{Z})^{\times} \\
& \cong  {\rm Gal}(L/K)\ \times \prod \limits_{p|m}\ (\mathbb{Z}/p^{v_p(n)} \mathbb{Z})^{\times}.
\end{split}
 \]
        
        \noindent  Hence by Corollary \ref{galois}, $M/K$ is obtained by nontrivial strong cluster magnification from $L/K$ through $F/K$ where $F= \mathbb{Q}(\xi_{m}) = \mathbb{Q}(\xi_{n/l})$.

        \end{enumerate}
\end{ex}
\medskip

\section{Cluster Towers}
\label{cluster towers} 

\subsection{Cluster tower of a polynomial} \hfill

Let $f$ be an irreducible polynomial over $K$. Consider a complete set of representatives of clusters of roots of $f$ in $\bar{K}$. Let $(\beta_1, \beta_2,\ldots, \beta_s)$ be an ordering of this set where $s=s_K(f)$. Now consider the following cluster tower of fields
terminating at the splitting field $K_f$. 
\smallskip

Write the tower as 
\[ K \subseteq K(\beta_1) \subseteq K(\beta_1, \beta_2) \subseteq \dots \subseteq K(\beta_1, \beta_2, \ldots , \beta_s) = K_f.\]

In \cite{krithika2023root}, the notions of degree sequence and length of such tower are defined as follows.
\smallskip

The length of tower is number of distinct fields in the tower and the degrees of these distinct fields over $K$ form the degree sequence.
\smallskip

\smallskip

\begin{ex} \label{Sn example}
As noted in \cite{krithika2023root}, if the Galois group of $f$ over $K$ is ${\mathfrak S}_n$ for $n>2$, we have $s=n$ and the cluster tower is given by 
\[ 
K \subsetneq K(\beta_1) \subsetneq K(\beta_1, \beta_2) \subsetneq  \dots \subsetneq K(\beta_1, \beta_2, \ldots , \beta_{n-1})= K(\beta_1, \beta_2, \ldots , \beta_{n}) = K_f, \]
 with degree sequence
$ (n, n(n - 1), n(n - 1)(n - 2), \ldots, n!/2, n!) =
 (^nP_1, ^nP_2 , \ldots, ^nP_{n-1})
$. So in this case the degree sequence is independent of ordering of $\beta_i$
’s. The length of the tower is $n$.

\end{ex}

\medskip

\noindent {\bf An important example:} A question was asked by Krithika and Vanchinathan in \cite{krithika2023root}: Is the degree sequence in general independent of the ordering of the representatives of the clusters of roots? We describe the following example that answers this question negatively.\smallskip

First we mention some easy to verify properties of Euler's totient function $\phi$.\smallskip

\begin{prop}
\label{Euler}
\hfill \medskip

{\it \noindent Suppose $l$ and $n$ are positive integers such that $l |n$ with $n = lm$. Consider their prime factorisations $n=\prod\limits_{p} p^{v_p(n)}$ and $l=\prod\limits_{p}  p^{v_p(l)}$ with $v_p(l)\leq v_p(n)$ for every prime $p$ (here $v_p$ is usual $p$-adic valuation). Then

\begin{enumerate} 
    \item $\phi (n)/ \phi (l)= m \phi(k) / k$ where $k=\prod\limits_{p\nmid l} p^{v_p(n)}$.\smallskip

\item $k|m$ and hence, $\phi (l) | \phi (n)$.\smallskip

\item $\phi(n)=\phi(l)$ if and only if $n=l$, or $l$ is odd and $n=2l$. \smallskip

\item $\phi (n)/ \phi (l)=m$ if and only if $n$ and $l$ have same prime factors.

 \end{enumerate}
 }
\end{prop}
\smallskip

\begin{ex}

\label{cluster tower example}
 Let $n\geq 6$.  Fix $\bar{\mathbb{Q}}$ to be an algebraic closure of $\mathbb{Q}$. Fix $b$ to be a primitive $n$-th root 
of unity in $\bar{\mathbb{Q}}$. Let $c$ be a positive rational number such that $f=x^n-c$ is
  an irreducible polynomial over $\mathbb{Q}$. (In particular, $c = p$, a prime works for any $n$ by Eisenstein criterion). Let $a = c^{1/n}$ be the positive real root of $f$. Hence the roots of $f$ are given by $a, ab, \ldots, ab^j, ab^{j+1},\ldots, ab^{n-1}$. We observe that, when $n$ is odd, $r=1, s=n$. When $n$ is even, the roots appear in pairs $\{ \alpha, -\alpha\}$ and thus $r=2,s=n/2$.\medskip

By \cite[Prop. 1 and Thm.\ A ]{jacobson1990galois}, the Galois group $G$ of splitting field of $f$ is isomorphic to $\mathbb{Z}/n \mathbb{Z} \rtimes (\mathbb{Z}/n \mathbb{Z})^{\times}$  if and only if $n$ is odd or, $n$ is even with $\sqrt{c}\not \in \mathbb{Q}(b)$  if and only if  $\mathbb{Q}(c^{1/n})\cap \mathbb{Q}(b)=\mathbb{Q}$. Assume $n$ to satisfy these conditions. Hence, in particular the order of $G$ is $n\phi(n)$. \medskip

Further, assume that $n$ is composite and $l|n$ with $n = ml$, where for the cases $n$ odd or $n\equiv 0\ (\text{mod}\ 4)$, we assume $2<l<n$; and for the case $n\equiv 2\ (\text{mod}\ 4)$ we assume $2<l<n/2$.  Because of our assumptions, $ab^{m}\neq \pm a, \pm ab$ and $1<\phi(l)<\phi(n)$ by (3) of Prop.\ \ref{Euler}. Also, $\phi (l) | \phi (n)$ by (2) of Prop.\ \ref{Euler}. Since $\mathbb{Q}(b)/\mathbb{Q}$ is Galois, by Thm.\ 2.6 in \cite{conrad2023galois}, we get 
\[{{\rm Gal}} (\mathbb{Q} (a, b^{m})/\mathbb{Q}(a))\cong {{\rm Gal}} (\mathbb{Q} (b^{m})/\mathbb{Q})\cong (\mathbb{Z}/l \mathbb{Z})^{\times}.\]
(For example, $c=2,n=9,l=3$ and $c=3,n=8,l=4$ work.)\smallskip

   Because of our assumptions, we have $m +1\leq s$. Let the representatives $\beta_i$s of clusters of roots of $f$, for $1\leq i\leq s$ be given by
   \[ \beta_3=ab^{m},\ \beta_{m+1}=ab^2,  \ \text{and}\ \beta_i=ab^{i-1}\ \text{for} \ i\neq 3,m+1  \  .\]
   
    Consider the following cluster towers:

    $$\mathbb{Q}\subsetneq \mathbb{Q} (\beta_1)\subsetneq \mathbb{Q}(\beta_1,\beta_2)=\mathbb{Q}_f, $$ 
with degrees $n$ and $n\phi(n)$ and length of tower $=3$ and
    $$\mathbb{Q}\subsetneq \mathbb{Q}(\beta_1)\subsetneq \mathbb{Q}(\beta_1,\beta_3) \subsetneq \mathbb{Q}(\beta_1,\beta_3, \beta_2)=\mathbb{Q}_f,$$ 
    with degrees $n, n\phi(l)$ and $n\phi(n)$ and length of tower $=4$.\medskip

    This example shows us that not only the degree sequence is not independent of the ordering of the $\beta_i$’s but length of tower is also not independent of the ordering of the $\beta_i$’s.

\end{ex}

  
\medskip

\subsection{Group theoretic formulation of cluster towers} \label{group-cluster-tower} \hfill

 Let the notations be as earlier in Sec.\ \ref{sec-SCM}. Let 
 $(\beta_1, \beta_2,\dots, \beta_s)$ be a fixed ordering of a complete set of representatives of the clusters of roots
of an irreducible polynomial $f$ over $K$ in $\bar{K}$, where $s = s_K(f)$. We have the cluster tower: 
\[ K \subset K(\beta_1) \subset K(\beta_1, \beta_2) \subset \dots \subset K_f.\]

Let $G={\rm Gal}(K_f/K)$. For each $1 \leq i \leq s$, let $H_i$ be the subgroup of $G$ that fixes $K(\beta_i)$. Let $\beta_1=\beta$ and $H_1=H$. Let $\sigma_i$ be isomorphism from $K(\beta)$ to $K(\beta_i)$ mapping $\beta$ to $\beta_i$ (hence $\sigma_1=id$). Then $H_i=\sigma_i H \sigma_i^{-1}$.\smallskip

Let $K_m=K(\beta_1,\beta_2,\ldots, \beta_m)$ and  $J_m: =(\bigcap \limits_{1 \leq i \leq m}\  H_i)$ be 
the subgroup of $G$ that fixes $K_m$.
\smallskip

Let  $m_1< m_2< \dots < m_l$ be all the indices $i > 1$ such that $J_i\neq J_{i-1}$. The length of above cluster tower is $l+2$.
Here, $m_l$ is smallest index $i$ such that $K_{i}=K_f$. \smallskip
 
 The degree sequence is $a_0 = n, a_1, a_2,\ldots, a_{l} = |G|$ with $a_i=[G:J_{m_{i}}]$ for all $i \geq 1$. \smallskip
 
Let $m_0=1$. Now for $i\geq 1$, $\dfrac{a_i}{a_{i-1}}=[K_{m_i}: K_{m_{i-1}}]$ and $K_{m_{i-1}}= K_{m_i -1}$ and number of roots of $f$ contained in $K_{m_i - 1}$ is $\geq (m_i -1)r$ (where $r=r_K(f)$). Also $K_{m_i}=K_{m_i -1}(\beta_{m_i})$. \smallskip

Thus $\dfrac{a_i}{a_{i-1}}= [K_{m_i}: K_{m_i -1}] \leq \left( n-(m_{i}-1)r \right)$ for all $i\geq 1$. Hence we have 
 
 \[ |G|\leq n\ \prod \limits_{1\leq i\leq l}\ \left( n-(m_{i}-1)r \right).\]
 
\begin{rem}
We have $a_1=n(n-1) \implies r_K(f)=1$. The converse is not true. Consider Example \ref{cluster tower example} for $n$ odd and composite. Then $r_K(f)=1$
and $a_1\leq n\phi(n)< n(n-1)$.
 \end{rem}

  Now since $H_i=\sigma_i H \sigma_i^{-1}$ for all $i$, we have that $N_G(H_i)=\sigma_i N_G(H) \sigma_i^{-1}.$ From 
 Thm.\ \ref{Perlis} (3), we have $r_K(K_m)=[N_G(J_m):J_m]$.
\medskip

\begin{ex} Let $f$ be an irreducible polynomial such that $|G|=np$ and $|H|=p$ where $p\nmid n$. Then length of cluster tower is $3$ and degree sequence is $n, np$. Both degree sequence and length of cluster tower are independent of the ordering of the $\beta_i$’s. As a particular case, $G = \mathfrak A_4$ for a degree-$4$ irreducible polynomial $f$.
\end{ex}

\section{Root Capacity}\label{root capacity}

\subsection{The group of automorphisms of finite extensions}
\hfill

Let $ F_1/F_2$ be a finite extension of fields. Let ${\rm Aut}(F_1/F_2)$ denote the group of $F_2$-automorphisms of $F_1$. In this section, we describe some of the facts about this group and later use it to prove some results about root clusters.\smallskip

\begin{prop} 
\label{r divide}
\hfill \medskip

{\it \noindent
Let $L/K$ and $M/L$ be extensions. Then \smallskip

    \begin{enumerate}
         \item ${\rm Aut}(M/L)$ is a subgroup of ${\rm Aut}(M/K)$. Hence, $r_L(M)~|~r_K(M)$.\smallskip

        \item Suppose $\sigma|_L \in {\rm Aut}(L/K)$ for any $\sigma \in {\rm Aut}(M/K)$.
         Then $ {\rm Aut}(M/L)\ \unlhd \ {\rm Aut}(M/K)$ and $r_K(M) ~|~ (r_L(M)\ r_K(L))$. \smallskip

        \item  Suppose $\sigma|_L \in {\rm Aut}(L/K)$ for any $\sigma \in {\rm Aut}(M/K)$. Then any $\lambda\in {\rm Aut}(L/K)$ can be extended to $\tilde{\lambda}\in {\rm Aut}(M/K)$ $\iff$ $r_K(M)= r_L(M) r_K(L)$. In this case $M/K$ is obtained by weak cluster magnification of $L/K$ with magnification factor $r_L(M)$.
         
    \end{enumerate}
    }
\end{prop}

\begin{proof}
     From Thm.\ \ref{Perlis} (3), $r_K(L)=|{\rm Aut}(L/K)|$. 
     Now (1) is easy to see.\smallskip

Proof of (2):
Suppose we have $\sigma|_L \in {\rm Aut}(L/K)$ for any $\sigma \in {\rm Aut}(M/K)$. Then we can define the a homomorphism $\Phi : {\rm Aut}(M/K)\rightarrow {\rm Aut}(L/K)$ by mapping $\sigma$ to $\sigma|_L$. Hence, ${\rm ker}(\Phi)={\rm Aut}(M/L)$ and so ${\rm Aut}(M/L)\unlhd {\rm Aut}(M/K)$. Also, ${\rm Aut}(M/K)/{\rm Aut}(M/L)\xhookrightarrow{}{\rm Aut}(L/K)$. Thus 
$r_K(M)~ | ~(r_L(M)\ r_K(L))$.\smallskip

Proof of (3): This corresponds to the homomorphism $\Phi$ being surjective.
\end{proof}
\smallskip

\begin{rem}
    The homomorphism $\Phi$ in above proof is not surjective in general. For example, when $K=\mathbb{Q}, L=\mathbb{Q}(\sqrt[]{2}), M=\mathbb{Q}(\sqrt[4]{2})$, then $\Phi$ is not surjective.
\end{rem}
\smallskip

\begin{prop}\label{N Aut}
\hfill \medskip

{\it \noindent
    Consider extensions $L/K$ and $M/L$. Then for $\sigma \in {\rm Aut}(M/K)$ we have 
    \[ \sigma \in N_{{\rm Aut}(M/K)}({\rm Aut}(M/L)) \iff \sigma|_{M^{{\rm Aut}(M/L)} } \in {\rm Aut}(M^{{\rm Aut}(M/L)}/K).\]
     Hence also 
    \[N_{{\rm Aut}(M/K)}({\rm Aut}(M/L))/ {\rm Aut}(M/L) \xhookrightarrow{}{\rm Aut}(M^{{\rm Aut}(M/L)}/K).\]
    }
\end{prop}

\begin{proof} 
We will mimic Perlis' proof of Thm. \ref{Perlis} (3) (see \cite[first proposition]{perlisroots}).\smallskip

For notational simplicity, (just for this proof) let  $G_1$, $G_2$ denote the groups ${\rm Aut}(M/K)$ and ${\rm Aut}(M/L)$, respectively.\smallskip

If $ \sigma \in N_{G_1}(G_2)$, then $\sigma\ G_2\ \sigma^{-1}= G_2$. Let $x\in M^{G_2}$. Since $\sigma\
    G_2 x = G_2 \sigma x$, we have $\sigma x = G_2 \sigma x$. Hence $\sigma x \in M^{G_2}$. \smallskip
    
    Conversely, suppose $\sigma \not \in N_{G_1}(G_2)$. Then there exists $\lambda \in G_2$ such that $\sigma^{-1} \lambda \sigma \not \in G_2$.  We know ${\rm {\rm Aut}}(M/L)={\rm Aut}(M/M^{G_2})$. Hence there exists $x\in M^{G_2}$ such that $\sigma^{-1} \lambda \sigma x \neq x$. That is $ \lambda \sigma  x \neq \sigma x$. Hence $\sigma x \not \in M^{G_2}$. Thus $\sigma|_{M^{G_2}} \not \in {\rm Aut}(M^{G_2}/K)$.\smallskip

    Then we can define the map $\Phi : N_{G_1}(G_2) \rightarrow {\rm Aut}(M^{G_2}/K)$ by mapping $\sigma$ to $\sigma|_{M^{G_2}}$. Hence, $\ker(\Phi)=G_2$ and also we have $$N_{G_1}(G_2)/ G_2 \xhookrightarrow{}{\rm Aut}(M^{G_2}/K).$$ 
    
    \vspace{-0.8 cm}
    
%
%
%
 \end{proof}

 \smallskip

\begin{cor}
\hfill \medskip

{\it \noindent
    Suppose $M/L$ is Galois. Then we can replace $M^{{\rm Aut}(M/L)}$ by $L$ in above Prop.\
    \ref{N Aut}. If additionally we have ${\rm {\rm Aut}}(M/L) \unlhd\ {\rm Aut}(M/K)$, then $\sigma|_L \in {\rm Aut}(L/K)$ for any $\sigma \in {\rm Aut}(M/K)$. This is converse to Prop.\ \ref{r divide} (2) under the condition $M/L$ is Galois.    }
    
\end{cor}

We will also give a direct proof of second part of this corollary without referring to above Prop.\
    \ref{N Aut}.\smallskip
    
\begin{proof}
    Let $\sigma \in {\rm Aut}(M/K)$. Now ${\rm {\rm Aut}}(M/\sigma(L))=\sigma {\rm Aut}(M/L) \sigma^{-1}$. Since ${\rm {\rm Aut}}(M/L)\unlhd {\rm Aut}(M/K)$, we have ${\rm {\rm Aut}}(M/\sigma(L))= {\rm Aut}(M/L)$. Also $|{\rm Aut}(M/\sigma(L))|= |{\rm Aut}(M/L)|=[M:L]$ since $M/L$ is Galois. Since $\sigma(L)$ is isomorphic to $L$, $[M:L]=[M:\sigma (L)]$. Hence, $M/\sigma(L)$ is also Galois. Thus $L=M^{{\rm Aut}(M/L)}=M^{{\rm Aut}(M/\sigma (L))}=\sigma (L)$. Therefore $\sigma|_L \in {\rm Aut}(L/K)$.
\end{proof}

By letting $M=\tilde{L}$ we have the following corollary, i.e., \cite[ first Prop.]{perlisroots}. 

 \begin{cor}
\label{perlis first prop}
 \hfill \medskip

{\it \noindent
     Let $G={\rm Gal}(\tilde{L}/K)$ and $H={\rm Gal}(\tilde{L}/L)$. Then we have $\sigma\in N_G(H) \iff \sigma|_L\in {\rm Aut}(L/K)$. Hence $\sigma H \mapsto \sigma|_L$ defines an isomorphism $N_G(H)/H \rightarrow {\rm Aut}(L/K)$.}
     
     \end{cor}
     
  \smallskip
\subsection{Root capacity} \hfill

We saw that the cluster size $r_K(\alpha)$ counts the number of roots appearing in the root cluster of $\alpha$ in $K(\alpha)$. We can ask for an analogous quantity associated to an extension $M/K$.\smallskip

\begin{defn}
\label{rho}
\hfill \medskip

\noindent
 Let $\alpha\in \bar{K}$ and let $f$ be minimal polynomial of $\alpha$ over $K$. For an extension $M/K$, let $\rho_K(M,\alpha)$ be the number of roots of $f$ that are contained in $M$. We call this quantity as the root capacity of $M$ with respect to $\alpha$ (with base field $K$ fixed).\smallskip

Let $L/K$ be an extension. By primitive element theorem $L=K(\alpha)$ for some $\alpha\in\bar{K}$. We define $\rho_K(M,L):=\rho_K(M,\alpha)$ which is well defined by the following proposition.  We call this quantity as the root capacity of $M$ with respect to $L$ (with base field $K$ fixed).
\end{defn}

\begin{prop}
\label{elegant}
\hfill \medskip

{\it \noindent
Let $\alpha,\beta \in \bar{K}$ and let $K(\alpha)=K(\beta)$. Then for any $M/K$, we have $\rho_K(M,\alpha)=\rho_K(M,\beta)$.
}
\end{prop}

\begin{proof}
    Let the degree-$n$ minimal polynomials of $\alpha$ and $\beta$ over $K$ be $f$ and $g$ respectively. Let $\{\alpha_i\}_{1\leq i \leq n}$ and $\{\beta_i\}_{1\leq i \leq n}$ be all the roots of $f$ and $g$ in $\bar{K}$ with $\alpha_1=\alpha, \beta_1=\beta$. Since $K(\alpha)=K(\beta)$, we have polynomials $\mu$ and $\lambda$ over $K$ with degrees $\leq (n-1)$ such that $\alpha=\lambda (\beta)$ and $\beta = \mu (\alpha)$. Now, $g(\beta)=0$. Hence $g(\mu (\alpha))=0$. Thus $f|g\circ \mu$. Hence $g(\mu (\alpha_i))=0$ for all $i$. Hence each $\mu (\alpha_i)=\beta_j$ for some $j$. \smallskip
    
    We also have $\alpha=\lambda \circ \mu (\alpha)$. Hence $f(x)| (\lambda \circ \mu (x)-x)$. Thus $\lambda \circ \mu (\alpha_i)=\alpha_i$ for all $i$. Hence $\mu$ is a bijection from the set of roots of $f$ to set of roots of $g$. By relabelling we can assume $\mu(\alpha_i)=\beta_i$ for all $i$. Thus $\lambda(\beta_i)=\alpha_i$ for all $i$. Hence for any $i$, we have $\alpha_i\in M \iff \beta_i\in M$. 
        \end{proof}

\begin{rem}
   Let $\alpha$ and $f$ be as in Def.\ \ref{rho}. Now, $\rho_K(K(\alpha),\alpha)=r_K(f)$. Thus the above proposition proves that given an extension $L/K$, the cluster size is the same for all irreducible polynomials that are the minimal polynomials of primitive elements of $L$ over $K$ (Corollary 1 in \cite{krithika2023root}). This proof is independent of Thm.\ \ref{Perlis} (3).
\end{rem}

\smallskip We prove some properties of root capacity.\smallskip

\begin{prop}
\label{rho properties}
\hfill \medskip

{\it \noindent
Let $M/K$ be extension of $L/K$. Then the following hold.
\smallskip

    \begin{enumerate}
   \item  $\rho_K(L,K)=1$ and $\rho_K(L,L)=r_K(L)$ and $\rho_K(M,L)\geq r_K(L)$ and $\rho_K(\tilde{L}, L)=[L:K]$.
\smallskip
       
  \item $\rho_K(M,L)= r_K(L)\implies$ $\sigma|_L \in {\rm Aut}(L/K)$ for any $\sigma \in {\rm Aut}(M/K)$.
    \end{enumerate}
    }
\end{prop}

\begin{proof}
The property (1) is easy to see.\smallskip
   
   Proof of (2): Let $L=K(\alpha)$ for $\alpha\in \bar{K}$ with $f$ as the minimal polynomial of $\alpha$ over $K$ with roots $\{\alpha_i\}_{1\leq i \leq n}$. Then $\rho_K(M,L)= r_K(L)$ is equivalent to the statement \quad $ \alpha_i\in M \iff \alpha_i\in L$. If $\sigma \in {\rm Aut}(M/K)$, then $\sigma|_L$ maps $\alpha$ to an $\alpha_i\in M$. Hence, $\sigma|_L (\alpha)\in L$. Thus $\sigma|_L \in {\rm Aut}(L/K)$. 
    \end{proof}

\smallskip

The following lemma is reformulation of \cite[Thm.\ 2] {perlis2004roots} for a  perfect base field $K$.

\begin{lem}
\label{perlis rho}

 {\it   Let $\alpha\in \bar{K}$ and let $f$ be minimal polynomial of $\alpha$ over $K$. Then for any $M/K$ we have $r_K(f)|\rho_K(M,\alpha)$. Hence for any $L/K$, we have $r_K(L)|\rho_K(M,L)$. That is $\rho_K(M,L)=a. r_K(L)$ where $0\leq a \leq s_K(L)$.}
\end{lem}
\smallskip

 The integer $a$ obtained in Lem.\ \ref{perlis rho} can be described as follows. 

\smallskip

\begin{prop} 
\hfill \medskip

{\it \noindent We have \smallskip
    \label{a}
     \begin{enumerate}
       \item $a$ is number of distinct fields inside $M\cap\tilde{L}$ isomorphic to $L$ over $K$.

  \smallskip     \item $a=|Z| $ with $Z= \{ 1\leq i\leq s 
    \ |\ \text{there exists}\ \sigma\in \sigma_i N_G(H)\ \text{with}\ \sigma(L)\subseteq M\cap \tilde{L}\}$ where $\sigma_i$ are coset representatives of $N_G(H)$ in $G$.

   \end{enumerate}
    }
\end{prop}

\begin{proof}
\hfill

\begin{enumerate}
\smallskip    \item In Lemma\ \ref{perlis rho}, we have that the set of roots of $f$ in $M$ is union of $a$ clusters of roots of $f$. From proof of Lemma\ \ref{L isom}, we are done.

 \smallskip   \item From the proof at the beginning of Sec.\ \ref{alternate proofs}, we have that $s_K(L)$ is number of distinct subgroups conjugate to $H={\rm Gal}(\tilde{L}/L)$ in $G={\rm Gal}(\tilde{L}/K)$ which is $[G:N_G(H)]$. From (1), we are done.
\end{enumerate} 

\vspace{-0.6 cm}

\end{proof}

\smallskip

\begin{ex}
    Let $L_k$ be as in proof of Thm.\ \ref{nPk}. Let $L=L_1$. So $r_K(L)=1,s_K(L)=n$ and $\tilde{L}=L_{n-1}$. Also $\rho_K(L_k,L)=k$ for $1\leq k\leq (n-2)$. Also $\rho_K(L_2,L)=r_K(L_2)=2$.
\end{ex}

\begin{ex}
    Let notation be as in subsection \ref{group-cluster-tower}. Let $L=K_1$ and $m_0=1$. Then for $i\geq 0$ and $m_i\leq m < m_{i+1}$, we have $\rho(K_m,L)=\rho(K_{m_i},L)\geq (m_{i+1}-1)r_K(L)$.
\end{ex}
\smallskip

\begin{prop}\label{rho compositum}
\hfill \medskip

\noindent
{\it    Let $M_1/K$ and $M_2/K$ be extensions of $L/K$ contained in $\bar {K}$. Then 
  \[ \rho_K(M_1M_2, L)\geq \rho_K(M_1,L)+\rho_K(M_2,L)-\rho_K(M_1\cap M_2, L).\]
  }
\end{prop}

\begin{proof}
    Consider the set of $s$ many representatives of clusters of minimal polynomial of a primitive element of $L/K$. Consider the subsets of representatives contained in $M_1$ and $M_2$ and let $a$ and $b$ be their respective cardinalities. Hence $\rho_K(M_1)=a.r_K(L)$ and $\rho_K(M_2)=b.r_K(L)$. Let $i$ be cardinality of intersection of these two sets. Hence $\rho_K(M_1\cap M_2, L)\geq i.r_K(L)$.
    Also since cardinality of union of these two sets is $a+b-i$. Hence $\rho_K(M_1M_2, L)\geq (a+b-i).r_K(L)$. Hence we are done. \end{proof}

\begin{ex}
    Let $p$ be a prime and let $a=\sqrt[12]{p}$ and $b=\xi_{12}$. Let $L=\mathbb{Q}(a)$ and $M_1=\mathbb{Q}(a, ab^2)$ and $M_2=\mathbb{Q}(a, ab^3)$. Thus $M_1 M_2 =\tilde{L}$. So $\rho_K(M_1M_2, L)=12, \rho_K(M_1,L)=6, \rho_K(M_2,L)=4$. Hence $\rho_K(M_1M_2, L)> \rho_K(M_1,L)+\rho_K(M_2,L)-\rho_K(M_1\cap M_2, L)$. The inequality is strict here.
\end{ex}

\medskip

Let $M/K$ be an extension of $L/K$ and $\rho_K(M,L)=a. r_K(L)$ for some $0\leq a \leq s_K(L)$. Let $L=K(\alpha)$ for some $\alpha\in \bar{K}$ with $f$ as the minimal polynomial over $K$. Relabel the roots $\{\alpha_i\}_{1\leq i \leq n}$ such that $\alpha_1=\alpha, \alpha_2,\ldots, \alpha_s$ forms a set of representatives of the clusters of roots of $f$ such that $\alpha_i\in M\iff 1 \leq i\leq a$. Let $L_M:=K(\alpha_1,\alpha_2,\ldots,\alpha_a)$. The roots of $f$ in $L_M$ and $M$ are same. We can observe that $L_M=\tilde{L}^{T}$ where $T=\bigcap \limits_{i\in Z}\ \sigma_i H \sigma_i^{-1}$ with notations as in Prop.\ \ref{a}.


\begin{prop}
\label{L_M}
\hfill \medskip

{\it \noindent
Let $M/K$ be extension of $L/K$. Let $L_M$ be as above. 
Then the following hold.\smallskip

 \begin{enumerate}

 \item $L_M$ is independent of choice of primitive root for $L/K$ that is $L_M$ is well defined for a given $M$ and $L$. \smallskip

 \item $L_M\subseteq M \cap \tilde{L}$.
 
 \smallskip    \item $\rho_K(L_M,L)=\rho_K(M,L)$. For $K \subset P\subset M$, we have $\rho_K(M,L)=\rho_K(P,L)$ if and only if $L_M \subset P$.
     
 \smallskip    \item $\rho_K(M,L)= r_K(L)\iff L=L_M$. 

 \smallskip    \item  $\rho_K(M,L_M)= r_K(L_M)$. Thus we have $\sigma|_{L_M} \in {\rm Aut}(L_M/K)$ for any $\sigma \in {\rm Aut}(M/K)$.

\smallskip     \item $M\cap \tilde{L}=L\implies \rho_K(M,L)=r_K(L)$. Also $M\cap \tilde{L}=K \implies \rho_K(M,L)=0$.

 \smallskip    \item If $M/K$ is obtained by strong cluster magnification from $L/K$ then $\rho_K(M,L)=r_K(L)$.

 \end{enumerate}
 }
 \end{prop}

\begin{proof}

Proof of (1): From the proof of Prop.\ \ref{elegant} it follows that if $K(\alpha)=K(\beta)$ then there is a relabelling of $\{\beta_i\}_{1\leq i \leq n}$ such that $K(\alpha_i)=K(\beta_i)$. Hence, for the above labelling of $\{\beta_i\}_{1\leq i \leq n}$, we have $K(\alpha_1,\alpha_2,\ldots,\alpha_a)=K(\beta_1,\beta_2,\ldots,\beta_a)$.\smallskip

Proofs of (2), (3), (4) follow from definition of $L_M$.

\smallskip Proof of (5): From Lem.\ \ref{perlis rho}, $\rho(M,L_M)=a.r_K(L_M)$ for $a\leq s_K (L_M)$. From Prop. \ref{a} (1), $a$ is number of distinct fields inside $M\cap \tilde{L}$ isomorphic to $L_M$ over $K$. By definition of $L_M$, we have $a=1$. The second assertion follows from Prop. \ref{rho properties} (2).\smallskip

Proof of (6): Since $M\cap \tilde{L}=L$, we use (2) to conclude that $L=L_M$. Hence by (4), we are done.\smallskip

Proof of (7): If the extension $M/K$ is obtained by strong cluster magnification from $L/K$, then we have $M\cap \tilde{L}=L$ (see Def.\  \ref{SCM} and Lem.\ \ref{jacob} ). 
\end{proof}
\smallskip

\begin{rem}
   The statement $L_M = M\cap \tilde{L}$ and the statement $\rho_K(M,L)=0 \implies M\cap \tilde{L}=K$ are not true in general. Let $K=\mathbb{Q}, L=\mathbb{Q}(\sqrt[4]{2})$ and $M=\mathbb{Q}(\sqrt[6]{2})$. Hence $\tilde{L}=\mathbb{Q}(\sqrt[4]{2}, \iota )$ and $M\cap \tilde{L}=\mathbb{Q}(\sqrt[]{2})$ and $L_M=\mathbb{Q}$.
\end{rem}
\smallskip

\begin{thm}\label{hint}
  {\it   If $M\cap \tilde{L}=L$ and $[M:L]=r_K(M)/r_K(L)$, then $M/L$ is Galois.  }
\end{thm}

\begin{proof}
Suppose $M\cap \tilde{L}=L$. Hence from Prop.\ \ref{L_M} (6), Prop.\ \ref{rho properties} (2) and Prop.\ \ref{r divide} (2), we have\\ $r_K(M)| (r_L(M)r_K(L))$. Since $r_K(M)=[M:L] r_K(L)$, we conclude that $r_L(M)=[M:L]$ and $M/L$ is Galois.
\end{proof}
\medskip

\section{Unique Chains for Extensions}\label{unique chains}

\subsection{Unique descending chains}\hfill

\begin{thm}

\label{unique descending chain}
  {\it   
Let $L/K$ be a nontrivial finite extension. Let $G={\rm Gal}(\tilde{L}/K)$ and $H={\rm Gal}(\tilde{L}/L)$.

\begin{enumerate}
    \item The extension $N=\tilde{L}^{N_G(H)}$ is the unique intermediate extension $N/K$ such that $L/N$ is Galois with degree $[L:N]=r_K(L)$. Hence the degree $[N:L]=s_K(L)$. \smallskip
    
       \item There is a unique strictly descending chain of subextensions
       \[ L=N_0\supsetneq N_1 \supsetneq N_2 \supsetneq \dots \supsetneq N_k\]
        such that for all $i\geq 1$, $N_{i-1}/N_i$ is Galois extension with degree $[N_{i-1}:N_i]=r_K(N_{i-1})$, with the chain terminating at $N_k$ for which $r_K (N_k)=1$. Hence the degree  $[N_i: K]=s_K(N_{i-1})$ for all $i\geq 1$ and $s_K(N_{k-1})=s_K(N_k)$.\smallskip

This unique strictly descending chain of subextensions corresponds to the unique strictly ascending chains of subgroups of $G$
\[H=H_0\properideal N_G(H_0)=H_1 \properideal N_G(H_1)=H_2 \properideal \dots \properideal N_G(H_{k-1})=H_k  \]
such that $N_G(H_k)=H_k$. Hence $r_K(N_i)=[H_{i+1}:H_{i}]$ and $s_K(N_i)=[G:H_{i+1}]$. \smallskip

    \item $L/K$ is obtained by weak cluster magnification from $N_k /K$ with magnification factor $r_K(L)$.\smallskip

    \item $N_k=K\iff N_{k-1}/K$ is Galois.\smallskip

     \item $r_K(L)=1\iff$ $ N_G(H)=H \iff$ the unique descending chain is singleton $L$.\smallskip

      \item $L/K$ is Galois $\iff N_G(H)=G$ $\iff$ the unique descending chain is $L\supsetneq K$.\smallskip

     \item $N_G(H)\properideal G\iff$ the unique descending chain is $L\supsetneq N\supsetneq K$.
\end{enumerate}
}
\end{thm}

\begin{proof}
\hfill
\smallskip
    
   Proof of (1):  Let $N=\tilde{L}^{N_G(H)}$. Thus $L/N$ is Galois as $H\unlhd N_G(H)$. We have $[L:N]=[N_G(H):H]$. Hence by Thm.\ \ref{Perlis} (3), $[L:N]=r_K(L)$. Suppose $N'/K$ is another intermediate extension such that $L/N'$ is Galois and $[L:N']=r_K(L)$. Since $L/N'$ is Galois, we have $H \unlhd {\rm  Gal}(\tilde{L}/N')\subseteq G$. Thus ${\rm Gal}(\tilde{L}/N')\subseteq N_G(H)$. Hence $N \subseteq N'$. Since $[L:N]=[L:N']$, we have $N=N'$.\smallskip

 Proof of (2):  Let $N_0=L$, From (1), we can inductively choose $N_i$ for each $i\geq 1$. Let $N_i/K$ be the unique intermediate extension of $N_{i-1}/K$ such that $N_{i-1}/N_{i}$ is Galois with degree $[N_{i-1}:N_i]=r_K(N_{i-1})$. The chain terminates since $L/K$ is finite.\smallskip
 
 Proofs of (3)-(6) are easy to see.
 \smallskip

 Proof of (7): 
 We have $N_G(H)=G\iff H\unlhd G\iff H=1$. We also have $N_G(H)\properideal G\implies$ $N_G(H)\neq H$. Since if $N_G(H)= H$, then by assumption, $H\properideal G$. Hence $H=1$, so $N_G(H)=G$ which contradicts the assumption. Hence, $N_G(H)\properideal G\iff$ $N_G(H)\neq H$ and $H\neq 1$ and $N_G(H)\unlhd G$
   $\iff$ $r_K(L)\neq1$ and $L/K$ is not Galois and $N/K$ is Galois $\iff$ the unique descending chain is $L\supsetneq N\supsetneq  K$. \end{proof}

   \smallskip

\begin{rem}  Alternate proof for Thm.\ \ref{unique descending chain} (1).

Let $N=L^{{\rm Aut}(L/K)}$. Hence $L/N$ is Galois and $[L:N]=|{\rm Aut}(L/K)|$. Hence by Thm.\ \ref{Perlis} (3), $[L:N]=r_K(L)$. Suppose $N'/K$ is another intermediate extension such that $L/N'$ is Galois and $[L:N']=r_K(L)$. Since $L/N'$ is Galois, we have $N'=L^{{\rm Aut}(L/N')}$ and $[L:N']=|{\rm Aut}(L/N')|$. Hence $|{\rm Aut}(L/N')|=|{\rm Aut}(L/K)|$. By Prop.\ \ref{r divide} (1), ${\rm {\rm Aut}}(L/N')\subset {\rm Aut}(L/K)$. Thus ${\rm {\rm Aut}}(L/N')= {\rm Aut}(L/K)$. Hence $N'=N$.\smallskip

Because of uniqueness, $\tilde{L}^{N_G(H)}=L^{{\rm Aut}(L/K)}$. This can also be seen in this way. Since $\tilde{L}^{N_G(H)}\subset \tilde{L}^H=L$ and $L^{{\rm Aut}(L/K)}\subset L$, by identifying $N_G(H)/H$ and ${\rm Aut}(L/K)$ through the map in Cor.\ \ref{perlis first prop}, we have $\tilde{L}^{N_G(H)}=L^{{\rm Aut}(L/K)}$.

\end{rem}

\begin{rem}
   Equivalently we can state Thm.\ \ref{unique descending chain} (1) as follows: There exists a unique intermediate extension $N/K$ such that $L/N$ is Galois of maximum possible degree. This is because, since $L/N$ is Galois, we have $[L:N]=|{\rm Aut}(L/N)|$ and we also have ${\rm {\rm Aut}}(L/N)\subset {\rm Aut}(L/K)$. Hence $[L:N]$ is bounded by $r_K(L)$.\smallskip

   This also gives an equivalent definition of $r_K(L)$ as the maximum possible degree of $L/N$ where $N$ is intermediate field of $L/K$ such that $L/N$ is Galois.
\end{rem}

\begin{prop}
\hfill \medskip

{\it \noindent
\label{symmetric sum}

Let $L/K$ be a nontrivial finite extension and $N$ be the unique intermediate extension for $L/K$ as in Thm. \ref{unique descending chain} (1). Suppose $L=K(\alpha)$ for a primitive element $\alpha\in \bar{K}$ with minimal polynomial $f$ over K such that $r_K(f)=r$ and let $\alpha=\alpha_1,\alpha_2,\dots, \alpha_r$ be the roots of $f$ contained in $L$. Then $N = K(t_1,t_2,\ldots,t_r)$ where $t_i$ are elementary symmetric sums of $\alpha_i$ for $1\leq i\leq r$.
}
\end{prop}

 \begin{proof}
    Let $h$ be the minimal polynomial of $\alpha$ over $N$. The degree of $h$ is $r=|Aut(L/N)|$. Since $Aut(L/N)=Aut(L/K)$, it follows that $h=\Pi_{1\leq i\leq r} (x-\alpha_i)=x^r-t_1 x^{r-1} + \dots + (-1)^r t_r$. \\
    Now $K(t_1,t_2,\ldots,t_r)\subset N$. Since $h$ is a polynomial over $K(t_1,t_2,\ldots,t_r)$ which $\alpha$ satisfies. we conclude that $h$ is the minimal polynomial of $\alpha$ over $K(t_1,t_2,\ldots,t_r)$. Thus $N = K(t_1,t_2,\ldots,t_r)$.\end{proof}
 \medskip

\subsection{Unique ascending chains}\hfill

We mention an analogue of Thm.~\ref{unique descending chain} in this subsection. The proof is similar.\smallskip

 First some notations: For any subgroup $H$ of a group $G$, we denote by $H^G$, the normal closure of $H$ in $G$, i.e., the intersection of all normal subgroups of $G$ that contain $H$.\smallskip

\begin{thm}
\label{unique ascending chain}
 {\it    
Let $L/K$ be a nontrivial finite extension. Let $G={\rm Gal}(\tilde{L}/K)$ and $H={\rm Gal}(\tilde{L}/L)$ be as earlier. 
\smallskip

\begin{enumerate}
    \item The extension $F=\tilde{L}^{H^G}$ is the unique intermediate extension $F/K$ such that $F/K$ is Galois with maximum possible degree.\smallskip



    \item We define the ascending index $t_K(L)$ of $L/K$ by $t_K(L):=[F:K]$. Let $u_K(L):=[L:F]$. Thus $t_K(L)\ u_K(L)=[L:K]$. We have $t_K(L)= [G:H^G]$ and $u_K(L)=[H^G:H]$. We also have $r_K(L)~ |~ t_K(L)\ r_F(L)$. \smallskip

    \item There exists a unique strictly ascending chain inside $L$, i.e.,
    \[ K=F_0\subsetneq F_1 \subsetneq F_2 \subsetneq \dots \subsetneq F_k \] such that for all $i\geq 1$, we have $F_{i}/F_{i-1}$ is Galois with maximum possible degree with the chain terminating at $F_k$ where $t_{F_k}(L)=1$. \smallskip

This unique strictly ascending chain of subextensions corresponds to the unique strictly descending chains of subgroups of $G$
\[G=G_0\properring H^{G_0}=G_1 \properring H^{G_1}=G_2 \properring  \dots \properring H^{G_{k-1}}=G_k\]
such that $H^{G_k}=G_k$. Hence $t_{F_i}(L)= [G_i: G_{i+1}]$ and $u_{F_i}(L)=[G_{i+1}:H]$.

    \smallskip

    \item $F_k=L\iff L/F_{k-1}$ is Galois.\smallskip

     \item $t_K(L)=1\iff H^G=G \iff$ the unique ascending chain is singleton $K$.\smallskip

      \item $L/K$ is Galois $\iff H^G=1\iff H^G=H\iff $ the unique ascending chain is $K\subsetneq L$.\smallskip

     \item $H\properideal H^G\iff H^G\neq G\ \text{and}\ H^G\neq H \ \text{and}\ H\unlhd H^G\iff$ the unique ascending chain is $K\subsetneq F\subsetneq  L$.
\end{enumerate}
}
\end{thm}

The following result connects the unique descending chain and unique ascending chain for a nontrivial finite extension $L / K$.\smallskip

\begin{prop}\label{link unique}
\hfill \medskip

{\it \noindent
    Let $L/K$ be a nontrivial finite extension. Let $N$ be as in Thm.\ \ref{unique descending chain} and $F$ be as in Thm.
     \ref{unique ascending chain}.

    \begin{enumerate}

      \item $N_G(H)=G \iff r_K(L)=[L:K] \iff t_K(L)=[L:K]\iff H^G=H$.
      
        \item $H\unlhd H^G\iff H^G\subset N_G(H) \iff  N\subset F$.

        \item $N_G(H)\unlhd G\implies H^G\subset N_G(H)$.
        
        \item $H^G=N_G(H)\implies N_G(H)\properideal G $ and $H\properideal H^G$.

        \item $H^G=N_G(H) \iff  N= F \iff $ the unique descending chain is $L\supsetneq N\supsetneq K$ and the unique ascending chain is $K\subsetneq F \subsetneq L$ and they both coincide. In this case we have $r_K(L) \ t_K(L)=[L:K]$ and $t_K(L)=s_K(L)$.
      
    \end{enumerate}
    }
\end{prop}

\begin{ex}

   For a degree $4$ non-Galois extension $L/K$, we have $N=F$. In particular, $K=\mathbb{Q}$, $L=\mathbb{Q}(\sqrt[4]{2})$ and hence $N=F=\mathbb{Q}(\sqrt[]{2})$.
\end{ex}

\medskip

\subsection{Interesting Examples}\label{Interesting Examples}\hfill

In this subsection, we will compute the unique ascending /descending chains for some examples.\medskip

 \begin{prop}\label{nth root}
    \hfill \medskip

{\it \noindent
Let the notation and conditions be as in Example \ref{cluster tower example}. Let $L=\mathbb{Q}(a)$ and $K=\mathbb{Q}$ and let $v_2$ be the standard $2$-adic valuation.

\begin{enumerate}
    \item Let $N_i=\mathbb{Q}(a^{2^i})$. Then the unique descending chain is $L=N_0\supsetneq N_1\supsetneq \dots \supsetneq N_{v_2(n)}$. \smallskip

    \item Let $F_i=\mathbb{Q}(a^{n/2^i})$. Then the unique ascending chain is $K=F_0\subsetneq F_1\subsetneq \dots \subsetneq F_{v_2(n)}$. 
\end{enumerate}
}
 \end{prop} 

\begin{proof}

 Identifying $G={\rm Gal}(\mathbb{Q}_f/\mathbb{Q})$ with $\mathbb{Z}/n \mathbb{Z} \rtimes (\mathbb{Z}/n \mathbb{Z})^{\times}$ we have $H= \{0\}\times (\mathbb{Z}/n \mathbb{Z})^{\times}\subseteq G$. Now $G$ has the semidirect product group law \[(\alpha,u)\cdot (\beta,v)=(\alpha + u\cdot \beta, uv)\] 
as in \cite{jacobson1990galois} where $u\cdot\beta$ is usual multiplication $u\beta$ in the ring $\mathbb{Z}/n \mathbb{Z}$. Thus $(\alpha,u)^{-1}=(- u^{-1} \alpha, u^{-1})$ and
\[(\alpha,u)\cdot(0,v)\cdot(\alpha,u)^{-1}=(\alpha-v\alpha,v)\]

Proof of (1):\smallskip

If $(\alpha, u)\in N_G(H) $ then $ \alpha=v\alpha$ for all $v\in (\mathbb{Z}/n \mathbb{Z})^{\times}$. \smallskip

If $n$ is odd then $2\in (\mathbb{Z}/n \mathbb{Z})^{\times}$ and in particular $2 \alpha = \alpha$. Thus $\alpha=0$ and $N_G(H)=H$. Hence $r_K(L)=1$. Hence by Thm.\ \ref{unique descending chain}(5), the unique descending chain is singleton $L$. \smallskip


If $n$ is even, then $(n/2)(1-v)=0$ for all $v \in (\mathbb{Z}/n \mathbb{Z})^{\times}$ as all these $v$ are given by odd integers mod $n$. Since $-1\in (\mathbb{Z}/n \mathbb{Z})^{\times}$, we get $2\alpha=0$. Hence $\alpha=0 \ \text{or}\ n/2$. Thus $N_G(H)$ can be identified with the set $\mathbb{Z}/2 \mathbb{Z} \times (\mathbb{Z}/n \mathbb{Z})^{\times}$ where $\mathbb{Z}/2 \mathbb{Z}$ is generated by the element $(n/2,1)$ of $G$ mapping $a$ to $-a$ and $b$ to $b$. Also $r_K(L)=2$. Hence $\tilde{L}^{N_G(H)}=\mathbb{Q}(a^2)=N_1$.\smallskip

Let $\tilde{N_1}$ be Galois closure of $N_1/K$. Hence $\tilde{N_1}=\mathbb{Q}(a^2,b^2)$. Also since 
$\mathbb{Q}(a)\cap \mathbb{Q}(b)=\mathbb{Q}$. Hence $\mathbb{Q}(a^2)\cap \mathbb{Q}(b^2)=\mathbb{Q}$. Hence ${\rm Gal}(\tilde{N_1}/K)$ can be identified with $\mathbb{Z}/(n/2) \mathbb{Z} \rtimes (\mathbb{Z}/(n/2) \mathbb{Z})^{\times}$. If $n/2$ is even, we can repeat the process for $N_1$ and get $N_2$. The process will terminate at $N_{v_2(n)}$ as $n/v_2(n)$ is odd. \smallskip

Proof of (2): \smallskip

Now $H^G$ is generated by elements of the form $(\alpha-v\alpha,v)$ with $\alpha\in \mathbb{Z}/n \mathbb{Z}, v\in (\mathbb{Z}/n \mathbb{Z})^{\times}$.\smallskip

If $n$ is odd then $2\in (\mathbb{Z}/n \mathbb{Z})^{\times}$. Thus $(\alpha-2\alpha,v)=(-\alpha,v)\in H^G$ for any $\alpha,v$. Thus $H^G=G$. Hence $t_K(L)=1$. Hence by Thm.\ \ref{unique ascending chain}(5), the unique ascending chain is singleton $K$. \smallskip

If $n$ is even then all $v\in (\mathbb{Z}/n \mathbb{Z})^{\times}$ are given by odd integers mod $n$. Thus all $\alpha-v\alpha$ are even mod $n$. Since $-1\in (\mathbb{Z}/n \mathbb{Z})^{\times}$, we get $(2\alpha,v)\in H^G$ for any $\alpha,v$. Thus $H^G=\mathbb{Z}/(n/2) \mathbb{Z} \rtimes (\mathbb{Z}/n \mathbb{Z})^{\times}$ where $\mathbb{Z}/(n/2) \mathbb{Z}$ is generated by the element $(2,1)$ of $G$ mapping $a$ to $ab^2$ and $b$ to $b$. Also $t_K(L)=2$. Hence $\tilde{L}^{H^G}=\mathbb{Q}(a^{n/2})=F_1$.\smallskip

Let $L_1$ be Galois closure of $L/F_1$. Hence $L_1=\mathbb{Q}(a,b^2)$. Also since 
$\mathbb{Q}(a)\cap \mathbb{Q}(b)=\mathbb{Q}$. Hence $L\cap \mathbb{Q}(b^2)=\mathbb{Q}$. By Lemma \ref{jacob}, $L\cap F_1 (b^2)=F_1$. Hence ${\rm Gal}(L_1/F_1)$ can be identified with $\mathbb{Z}/(n/2) \mathbb{Z} \rtimes (\mathbb{Z}/(n/2) \mathbb{Z})^{\times}$. If $n/2$ is even, we can repeat the process for $F_1$ and get $F_2$. The process will terminate at $F_{v_2(n)}$ as $n/v_2(n)$ is odd. \end{proof}

\smallskip

 \begin{rem} \label{nth root remaining case}

 The Part (1) above is also true for the case not covered, that is $n$ is even and $\sqrt{c}\in \mathbb{Q}(b)$. The following proof works in general. We know that $N_1=L^{{\rm Aut}(L/K)}$. Since $n$ is even, ${\rm {\rm Aut}}(L/K)$ has $2$ elements mapping $a$ to $a$ and $a$ to $-a$. Hence $\mathbb{Q}(a^2)\subset L^{{\rm Aut}(L/K)}$. Now since $x^n-c$ is minimal polynomial for $a$ over $\mathbb{Q}$. Hence $a^2$ satisfies $x^{n/2}-c$. We claim that it is indeed the minimal polynomial for $a^2$ over $\mathbb{Q}$. If not, then let $f(x)$ be minimal polynomial for $a^2$ with degree $<n/2$. Then $a$ satisfies the polynomial $f(x^2)$ which has degree $<n$ which gives a contradiction. Hence $[\mathbb{Q}(a^2):\mathbb{Q}]=n/2$. Thus $N_1=\mathbb{Q}(a^2)$. Proceeding similarly we get the unique descending chain.
 \end{rem}

 \begin{ex}
%

  Consider the case in Prop.\ \ref{nth root} with $n\equiv 2\ \text{mod}\ 4$. Hence for $L=\mathbb{Q}(a)$ and $K=\mathbb{Q}$, we have $L\supsetneq N$ to be the unique descending chain with $N=\mathbb{Q}(a^{2})$ and $K\subsetneq F$ to be the unique descending chain with $F=\mathbb{Q}(a^{n/2})$. We can show that $L/K$ is obtained by strong cluster magnification from $N/K$ through $F/K$. \smallskip
  
  Now $F/K$ is clearly Galois. Since $(n/2,2)=1$. Hence $NF=L$. Thus $N\cap F=K$. Since $\mathbb{Q}(a)\cap \mathbb{Q}(b)=K$. Hence $\mathbb{Q}(a)\cap \mathbb{Q}(b^2)=K$. Thus by Lemma \ref{jacob}, $\mathbb{Q}(a)\cap \mathbb{Q}(a^2,b^2)=\mathbb{Q}(a^2)$. That is $\tilde{N}\cap L=N$. Hence $\tilde{N}\cap F=K$. 

\end{ex}

\medskip

\begin{thm}

 \label{Perlis unique}   
{\it 
Let $f$ be the irreducible polynomial over number field $K$ as in proof of Thm.\ \ref{n,r}  with given $n>2$ and $1<r<n$ with $r|n$  and $s=n/r$. Let $L/K$ be extension formed by adjoining a root of $f$ to $K$. Then the unique descending chain is $L\supsetneq N\supsetneq K$ and the unique ascending chain is $K\subsetneq F \subsetneq L$ and they both coincide.}

\end{thm}

\begin{proof}
Let $G={\rm Gal}(\tilde{L}/K)$ and $H={\rm Gal}(\tilde{L}/L)$. By construction, $G=(\Z/r\Z)^s\rtimes \Z / s\Z$ and $H=(\Z/r\Z)^{s-1}  \times \{0\} \times \{0\}$. Now $G$ has the semidirect product group law given by \[ ((a_1,\dots, a_s),b)\cdot((c_1,\dots,c_s),d)=((a_1,\dots, a_s) + (b\cdot (c_1,\dots, c_s)), b+d), \]
  where $b\cdot(c_1,\dots, c_s)=(c_{b+1},\dots, c_s,c_1,\dots, c_{b})$. \smallskip
  
  Also $((a_1,\dots,a_s),b)^{-1}=((-a_{s-b+1},\dots,-a_{s}, -a_{1},\dots, -a_{s-b}),-b)$. One can verify that
\[((a_1,\dots, a_s),b)\cdot((c_1,\dots,c_{s-1},0),0)\cdot ((a_1,\dots,a_s),b)^{-1}= ((c_{b+1},\dots, c_{s-1}, 0, c_1,\dots,c_b),0).\]
  
  Thus for $r<n$ (that is $s>1$), one can compute and show that $N_G(H)=H^G=(\Z/r\Z)^s\times \{0\}$. Hence by Prop.\ \ref{link unique} (5), we are done.\end{proof}

  \smallskip

\begin{rem}
    For $M/L/K$, the statements  $r_K(M) \geq r_K(L) $ and $r_K(L)|r_K(M) $ (weak cluster magnification) are not true in general. Consider $L/N/K$ as in Thm.\ \ref{Perlis unique}.
Here, $r_K(L) \geq r_K(N)\iff r^2\geq n$ and $r_K(N)|r_K(L) \iff n|r^2$. Thus in particular $n=6,r=2$ and $n=6,r=3$ give us counterexamples for the two statements.

\end{rem}

\medskip


 \begin{thm}\label{unique chain Lk}
{\it 
Let $f$ over $K$ be irreducible of deg $n$ with Galois group ${\mathfrak S}_n$ with roots $\alpha_i\in \bar{K}$ for $1\leq i\leq n$. For $1 \leq k \leq n-2$, let $L_k=K(\alpha_1,\dots, \alpha_k)$. 

\begin{enumerate}

    \item 

Let $N_k$ be the unique intermediate extension for $L_k/K$ as in Thm. \ref{unique descending chain} (1). Then we have the following.\smallskip

 \begin{enumerate}[(i)]

     \item $  N_k = K(t_1,t_2,\ldots,t_k)$ where $t_i$ are elementary symmetric sums of $\alpha_i$ for $1\leq i\leq k$.\smallskip

     \item 
   \noindent Case (a): Characteristic of $K$ $\neq 2$. For $k<n-1$, we have 
   $$N_k=K(t_1)=K(\alpha_1+\alpha_2\dots+\alpha_k).$$

\noindent Case (b) Characteristic of $K= 2$. For $k<n-1$ but $k\neq n/2$,  we have 
$$N_k=K(t_1)=K(\alpha_1+\alpha_2\dots+\alpha_k).$$

     \item for $k\neq 1, n/2$, the unique descending chain is $L_k\supsetneq N_k$. Also $r_K(L_1)=1$.\smallskip

     \item for $k=n/2$ and for field $K$ with characteristic $\neq 2$, the cluster size $r_K(N_k)=2$. The unique intermediate field for $N_k/K$ is $K(t_1(a-t_1))$ where $\alpha_1+\alpha_2+\dots+\alpha_n = a$.
 \end{enumerate}

 \medskip

 \item The unique ascending chain is singleton $K$. And $t_K(L_k)=1$.

 \end{enumerate}
 }   
\end{thm}

\begin{proof} \hfill \smallskip
\begin{enumerate}
    \item 
 \begin{enumerate}[(i)]
     \item  Now, $\tilde{L_k}=L_{n-1}=L_n$. As noted in proof of Thm.\ 3 in \cite{krithika2023root}, the subgroup $H_k \subseteq{\mathfrak S}_n$ fixing $L_k$ is isomorphic to $\mathfrak S_{n-k}$. From Lemma 3 in \cite{krithika2023root}, $N_{{\mathfrak S}_n} (H_k)\cong \mathfrak S_{n-k}\times \mathfrak S_k$ where $\mathfrak S_k$ permutes the $k$ roots $\alpha_1,\alpha_2,\dots,\alpha_k$ and $\mathfrak S_{n-k}$ permutes the other $n-k$ roots. Now $N_k=\tilde{L_n}^{N_{{\mathfrak S}_n} (H_k)}$. Hence ${\rm Gal}(L_k/N_k)\cong \mathfrak S_k $ and $[L_k:N_k]=k!$ and $[N_k:K]=\ ^n C_k$.
\smallskip
     
     Now $K(t_1,t_2, \dots ,t_k)\subseteq N_k$. Also $L_k$ is splitting field of polynomial $x^k-t_1 x^{k-1} + \dots + (-1)^k t_k$ over $K(t_1,t_2,\dots,t_k)$. Hence $[L_k : K(t_1,t_2,\dots,t_k)]\leq k!$. Thus $K(t_1,t_2,\dots,t_k)= N_k$.
  \smallskip

\item We give a proof for  case (a) only. The proof in the other case follows similarly. Now for any $l<n-1$, we have $\alpha_{i_{l+1}}\not \in K(\alpha_{i_1},\alpha_{i_2},\dots,\alpha_{i_l})$ and we also have $\alpha_{i_n}\in K(\alpha_{i_1},\alpha_{i_2},\dots,\alpha_{i_{n-1}})$ for distinct $i_m\leq n$. Since $t_1=\alpha_1+\alpha_2+\dots+\alpha_k$, we have that $t_1$ has at most $\ ^n C_k$ conjugates inside $\tilde{L_k}$ of the form $\alpha_{i_1}+\alpha_{i_2}+\dots+\alpha_{i_k}$ for distinct $i_m\leq n$.\smallskip

   We claim that number of conjugates is exactly $^nC_k$ that is, if $\{i_1,i_2,\dots, i_k\}$ and $\{j_1,j_2,\dots,j_k\}$ are distinct sets of $k$ numbers $\leq n$ then $\alpha_{i_1}+\alpha_{i_2}+\dots+\alpha_{i_k}\neq\alpha_{j_1}+\alpha_{j_2}+\dots+\alpha_{j_k}$. Assume on the contrary $\alpha_{i_1}+\alpha_{i_2}+\dots+\alpha_{i_k}=\alpha_{j_1}+\alpha_{j_2}+\dots+\alpha_{j_k}$.
   
   \smallskip

 Case 1 : Let $k<n/2$. Hence number of distinct $\alpha_i$ in the equation is at most $n-1$. After cancelling the common terms on both sides and then by taking all $\alpha_i$ on one side except say $\alpha_{i_l}$, we get $\alpha_{i_{l}}$ lies in field generated over $K$ by other $\leq n-2$ many $\alpha_i$. This gives a contradiction.\smallskip

 Case 2 : Let $k> n/2$. Now before cancelling we have $\leq n-k$ terms on RHS distinct than terms on LHS. Hence we have $\geq 2k-n$ common terms on both sides. Cancelling these common terms out leaves us with $\leq n-k$ terms on both sides. Since $k>n/2$, we have $n-k< n/2$. So we are back to Case 1.\smallskip

 Case 3 : Let $n$ be even and  $k=n/2$. We assume characteristic is not $2$ in this case. If we can cancel out even one term from both sides then we are back to Case 1. So assume that $\alpha_{i_1}+\alpha_{i_2}+\dots+\alpha_{i_{n/2}}=\alpha_{j_1}+\alpha_{j_2}+\dots+\alpha_{j_{n/2}}$ where all $i_m, j_m$ are distinct. We know that $\alpha_1+\alpha_2+\dots+\alpha_n = a \in K $. Hence $\alpha_{i_1}+\alpha_{i_2}+\dots+\alpha_{i_{n/2}} + \alpha_{j_1}+\alpha_{j_2}+\dots+\alpha_{j_{n/2}}=a$. Thus $2(\alpha_{i_1}+\alpha_{i_2}+\dots+\alpha_{i_{n/2}})=a$. Since characteristic is not $2$, we have $\alpha_{i_1}+\alpha_{i_2}+\dots+\alpha_{i_{n/2}}\in K$. Thus $\alpha_{i_{n/2}}\in K(\alpha_{i_1},\alpha_{i_2},\dots,\alpha_{i_{n/2 -1}})$ which is a contradiction.
 
\smallskip

 Hence $[K(t_1):K]=\ ^nC_k$. Now $K(t_1)\subset N_k$. Hence $K(t_1)= N_k$.\medskip

\item Case 1 : Let $k<n/2$. We claim that $r_K(N_k)=1$. Suppose $\alpha_{i_1}+\alpha_{i_2}+\dots +\alpha_{i_k}\in K(t_1)$ for at least one $i_j\geq k+1$. By similar argument as in Case 1 of (2), we get a contradiction. Hence the unique descending chain is $L_k\supsetneq N_k$.\smallskip

Case 2 : Let $k>n/2$. Thus $n-k<n/2$. So $r_K(N_{n-k})=1$. Hence $N_{{\mathfrak S}_n}({\mathfrak S}_k\times {\mathfrak S}_{n-k})={\mathfrak S}_k\times {\mathfrak S}_{n-k}$. By symmetry $N_{{\mathfrak S}_n}({\mathfrak S}_{n-k}\times {\mathfrak S}_{k})={\mathfrak S}_{n-k}\times {\mathfrak S}_{k}$. Thus $r_K(N_k)=1$.

\smallskip

\item Suppose $k=n/2$ and characteristic of $K$ $\neq 2$. Clearly 
$$(\alpha_{n/2+1}+\alpha_{n/2+2}+\dots + \alpha_n )\in K(\alpha_1+\alpha_2+\dots +\alpha_{n/2}).$$
 Suppose any other $\alpha_{i_1}+\alpha_{i_2}+\dots +\alpha_{i_k}\in K(t_1)$. Thus at least one  $i_l $ satisfies $1\leq i_l \leq k$. Hence we get that at least one of the $\alpha_i$ lies in field generated over $K$ by other $\leq n-2$ many $\alpha_j$, which is a contradiction. Hence $r_K(N_k) = 2$. By Prop.\ \ref{symmetric sum}, unique intermediate field for $N_k/K$ as in Thm. \ref{unique descending chain} (1) will be $K((t_1+a-t_1),t_1(a-t_1))=K(t_1(a-t_1))$. 
 \end{enumerate}
\smallskip

\item Proof of last assertion: We have that the smallest normal subgroup of ${\mathfrak S}_n$ containing ${\mathfrak S}_{n-k}$ is ${\mathfrak S}_n$ itself. This is because $k\leq n-2$ that is $n-k\geq 2$. We know that for $n=3$ and $n\geq 5$, the only proper non trivial normal subgroup of ${\mathfrak S}_n$ is ${\mathfrak A}_n$ and ${\mathfrak A}_n$ cannot contain ${\mathfrak S}_{n-k}$ for any $k\leq n-2$. We also know that for $n=4$, the only proper non trivial normal subgroups of ${\mathfrak S}_4$ are ${\mathfrak A}_4$ and $V_4$ (Klein four-subgroup) and ${\mathfrak A}_4$ and $V_4$ cannot contain ${\mathfrak S}_{4-k}$ for both $k=1,2$. Hence by Thm.\ \ref{unique ascending chain} (5), we are done.
\end{enumerate}
\vspace{-0.6 cm}
\end{proof}

\smallskip

\begin{rem}
    From the above theorem, in the case $k\neq 1, n/2$, the unique descending chain is $L_k\supsetneq N_k$. Hence $L_k/K$ is obtained by nontrivial weak cluster magnification from $N_k/K$ with magnification factor $[L_k:N_k]$. But $L_k/K$ is not obtained by nontrivial strong cluster magnification from $N_k/K$, since the last assertion of above theorem implies that there doesn't exist a Galois $F/K$ contained in $L_k$ as in Def \ref{SCM} with $[F:K]=[L_k:N_k]$.
    
\end{rem}

\begin{rem}
%
   In view of the examples discussed in Sec.\ \ref{Interesting Examples}, we can see that the converse of Thm.\ \ref{hint} 
   is not true. \smallskip

\begin{enumerate}

 \item  Let $M=\tilde{L}$ where $L/K$ is not a Galois extension. Then $M/L$ is Galois but $M\cap \tilde{L}\neq L$ and $[M:L]\neq r_K(M)/r_K(L)$.\smallskip

    \item Consider the case in Thm.\ \ref{Perlis unique}. We have $L/N$ is Galois and $L\cap \tilde{N}=N$ but $[L:N]\neq r_K(L)/r_K(N)$.\smallskip

    \item Consider the case in Thm.\ \ref{unique chain Lk} for $n\geq 5$ and $k=2$ and  Characteristic of $K \neq 2$. Since $N_k=K(\alpha_1+\alpha_2)$. Thus $\alpha_1+\alpha_3,\alpha_2+\alpha_3\in \tilde{N_k}$. Hence $\alpha_1-\alpha_2\in \tilde{N_k}$ and thus $\alpha_1,\alpha_2\in \tilde{N_k}$. That is $L_k\subset \tilde{N_k}$. So $L_k/N_k$ is Galois and $[L_k:N_k]=r_K(L_k)/r_K(N_k)$ but $\tilde{N_{k}}\cap L_k\neq N_k$.

    \end{enumerate}

    \smallskip

On a similar note, the conclusion of Thm.\ \ref{hint} does not hold if  
we remove any one of the two conditions in its hypothesis.\smallskip

\smallskip
    \begin{enumerate}

    \item Let $M/K$ be a nontrivial non Galois extension. Let $L=K$. Then $M\cap \tilde{L}=L$ but $[M:L]\neq r_K(M)/r_K(L)$ and $M/L$ is not Galois.\smallskip

    \item Consider the case in Ex \ref{j=n-k} for odd $n\geq 5$ and $k=(n+1)/2$ and $j=n-k$. Hence $j=(n-1)/2=k-1$. Thus $L_k=L_j(\alpha_k)$. One can verify that the minimal polynomial of $\alpha_k$ over $L_j$ has degree $n-k+1$ and has the roots $\alpha_k,\alpha_{k+1},\dots, \alpha_n$. Also $\alpha_i\not \in L_k$ for $i>k$. Thus $[L_k:L_j]=r_K(L_k)/r_K(L_j)$ but $\tilde{L_{j}}\cap L_k\neq L_j$ and $L_k/L_j$ is not Galois.
    \end{enumerate}

\end{rem}

 \medskip

\section{Base Change theorems}\label{base change}

All along we have been discussing clusters for irreducible polynomial over some fixed perfect base field $K$, or equivalently for
extension $L/K$. Now we would like to define a notion of `base change'.

\begin{defn} \hfill \medskip

 \noindent
We say that a finite extension $L/K$ has a `base change of $K$ by an extension $K' / K$' if $K'/K$ is a finite Galois
extension such that $\tilde{L}$ and $K'$ are linearly disjoint over $K$.
\end{defn}

\begin{rem}
    Consider $M/L/K$. If $M/K$ has a base change of $K$ by $K'$ then $L/K$ also has a base change of $K$ by $K'$. 
\end{rem}

Throughout this section we consider extensions $L/K$ which have base change of $K$ by $K'$ and study the properties of cluster magnification and other phenomena for corresponding extensions $LK'/K'$. 
 
 \smallskip

The following result establishes a certain conditional transitivity of base change.

 \begin{prop}
 \hfill \medskip

{\it \noindent
  Consider an extension $L/K$ which has a base change of $K$ by $K'$. Further, consider the extension $LK'/K'$ which has a base change of $K'$ by $K''$ such that $K''/K$ is Galois. Then the extension $L/K$ has a base change of $K$ by $K''$.
  }
 \end{prop}

 \begin{proof}
     We have $K'/K$ Galois with $\tilde{L}\cap K'=K$ and $K''/K'$ Galois with $\tilde{LK'}\cap K'' = K'$. Observe that $\tilde{L}=\tilde{L}\cap \tilde{LK'}$. Thus $\tilde{L}\cap K''= (\tilde{L}\cap \tilde{LK'}) \cap K''= \tilde{L}\cap( \tilde{LK'}\cap K'')= \tilde{L}\cap K'= K$. Since $K''/K$ is Galois, we are done.
 \end{proof}

\smallskip

 \subsection{A base change theorem for strong and weak cluster magnification}\hfill
 
 First we see a result about field extensions. The following lemma is a special case of \cite[Prop.\ 2.5]{bhagwat2024right}. 
\begin{lem}
    \label{shubham}

{\it    Let $E_1,E_2,E_3$ be Galois over $K$ contained in $\bar{K}$. Suppose each pairwise intersections of $E_i$s is $K$, i.e., \[ E_1\cap E_2=E_2\cap E_3=E_3\cap E_1=K.\] Then 
    \[ E_1 E_2 \cap E_3 =K \iff E_1 E_3 \cap E_2 E_3 =E_3.\]
    }
\end{lem}

\smallskip

\begin{thm}

\label{basechange}
{\it 
 Suppose $M/K$ is obtained by nontrivial strong cluster magnification from an $L/K$ with magnification factor $d$. Let $M/K$ have a base change of $K$ by $K'$. Then $MK'/K'$ is obtained by nontrivial strong cluster magnification from $LK'/K'$ with the same magnification factor $d$.\smallskip
 
  Furthermore,
  $MK'/K$ is obtained by strong cluster magnification from $M/K$ and $LK'/K$ is obtained by strong cluster magnification from $L/K$ and these are non trivial if $[K':K]>1$.
  }
\end{thm}

\begin{proof}

Proof of second assertion essentially follows from definition of cluster magnification. We prove the first assertion. Suppose $M/K$ is obtained by strong cluster magnification from an $L/K$. We have $L/K$ and $F/K$ as in Def.\ \ref{SCM}. Let $M':= MK'$. Since $\tilde{M}$ and $K'$ are linearly disjoint over $K$, we have $\tilde{M}\cap K'=K$. Since $\tilde{M}=\tilde{L}F$, we have $\tilde{L}\cap K'=K$, $L\cap K'=K$ and $F\cap K'=K$. \smallskip
 
 Now we check the three conditions of Def. \ref{SCM} for $M'/K'$.
\begin{enumerate}

\item Let $L':=LK'$. Since $K'/K$ Galois and $L\cap K'=K$, we have $[L':K']=[L:K]=n>2$. \smallskip

\item Let $F':=FK'$. Since, $F/K$ is Galois, $F'/K'$ is Galois. Also, $[F':K']=[F:K]=d>1$. Let $\tilde{L'}$ be Galois closure of $L'$ over $K'$ in $\bar{K}$. Now, $N :=\tilde{L}K'$ is Galois over $K'$ and $L' \subset N$. Hence, $\tilde{L'}\subset N$. Since $F$ is linearly disjoint with $\tilde{L}$ over $K$, we have $\tilde{L}\cap F=K$. Also $\tilde{L}\cap K'=K$ and $F\cap K'=K$. Since $\tilde{L}, F$ and $K'$ are Galois over $K$. From Lem.\ \ref{shubham} for $E_1=\tilde{L} ,E_2=F , E_3=K'$, we have $\tilde{L} F\cap K'=K\iff  \tilde{L} K' \cap F K'=K'$. So we have $N\cap F'=K'$. Thus, $\tilde{L'}\cap F=K'$. Hence $\tilde{L'}$ and $F'$ are linearly disjoint over $K'$.\smallskip

\item $L'F'=(LK')(FK')=(LF)K'=MK'=M'$.
\end{enumerate}
\vspace{-0.6 cm}
\end{proof}

\smallskip

\begin{rem}
    The above theorem can be reformulated in this way: If $M/K$ is obtained by strong cluster magnification from $L/K$ through $F/K$ and $M'/K$ is obtained by strong cluster magnification from $M/K$ through $K'/K$, then $M'/K'$ is obtained by strong cluster magnification from $LK'/K'$ through $FK'/K'$. \smallskip

    We also have $\rho_K (M',L)=r_K(L)$. This is because $M'\cap \tilde{M}=M$ and $M\cap \tilde{L}=L$ which imply $M'\cap \tilde{L}=L$.
\end{rem}

\begin{lem}
\label{isomorphism}
{\it 
The extension $M/K'$ is $K'$-isomorphic to $LK'/K'$ $\iff$ $M=L_1K'$ where $L_1/K$ is $K$-isomorphic to $L/K$. Further in this case, the extension $L_1$ is unique and is given by $L_1=M\cap \tilde{L}$.}
\end{lem}

\begin{proof}
    Suppose $M/K'$ is isomorphic to $LK'/K'$. Let $\sigma : LK'\rightarrow M$ be the isomorphism such that $\sigma|_{K'}=id_{K'}$. Let $\sigma(L)=L_1$. Hence $M=L_1K$. Since $\sigma|_K=id_K$, it follows that $L_1/K$ is isomorphic to $L/K$.\smallskip   
    
Conversely, suppose $L_1/K$ is isomorphic to $L/K$. Let $\lambda : L\rightarrow L_1$ be the isomorphism such that $\lambda|_K=id_K$. Let $\tilde{\lambda}:LK'\rightarrow L_1K'$ be such that $\tilde{\lambda} (l)=\lambda (l)$ for all $l\in L$ and $\tilde{\lambda}(k')=k'$ for all $k'\in K'$. Let $\{ l_i \}_{1 \leq i \leq [L:K]}$ be a $K$-basis for $L$.  Hence any element of $LK'$ is of the form $\sum \limits_{i} l_i k'_i$ for $k'_i\in K'$.\smallskip

    Suppose $\sum_i l_i k'_i=0$. Since $l_i\in L\subset \tilde{L}$ are linearly independent over $K$, and $\tilde{L}$ and $K'$ are linearly disjoint over $K$; it follows by \cite[Def. 20.1]{morandi2012field}, we have that $\{ l_i \}_{1 \leq i \leq [L:K]}$ are linearly independent over $K'$. Thus $k'_i=0$ for all $i$. Now $\tilde{\lambda} (\sum_i l_i k'_i)= \sum_i \lambda (l_i) k'_i=0$. Hence $\tilde{\lambda}$ is well defined field isomorphism with $\tilde{\lambda}|_{K'}=id_{K'}$.\smallskip

    As $L_1\subset \tilde{L}$ and by Lemma \ref{jacob}, we have $\tilde{L}\cap L_1K'=L_1$. Thus the uniqueness of $L_1$ follows.
\end{proof}
\smallskip

\begin{cor}

\label{isom K}

\hfill \medskip

{\it \noindent
    $M/K$ is $K$-isomorphic to $LK'/K$ $\iff$ $M=L_1K'$ where $L_1/K$ is $K$-isomorphic to $L/K$. (In this case such $L_1$ is unique).
    \smallskip
    
   (In particular, $M/K$ is isomorphic to $LK'/K$ $\iff$ $M/K'$ is isomorphic to $LK'/K'$.)}
\end{cor}

\begin{proof}
     Suppose $M/K$ is isomorphic to $LK'/K$. Let $\sigma : LK'\rightarrow M$ be the isomorphism such that $\sigma|_{K}=id_{K}$. Since $K'/K$ is Galois, it follows that $\sigma(K')=K'$. Let $\sigma(L)=L_1$.  Hence $M=L_1K$ and $L_1/K$ is isomorphic to $L/K$.\smallskip
     
     Conversely, suppose $L_1/K$ is isomorphic to $L/K$. Hence by Lemma \ref{isomorphism}, $L_1K'/K'$ is isomorphic to $LK'/K'$. Thus $L_1K'/K$ is isomorphic to $LK'/K$.
\end{proof}

\medskip

\noindent{\bf An alternate proof for Cluster Magnification theorem Thm.\ \ref{Vanchinathan} \cite[Thm. \ 1]{krithika2023root}:}\smallskip

\begin{proof}
By Lemma \ref{jacob}, since $F/K$ is Galois, we have
$$\tilde{L}\cap F=K\iff \tilde{L}\cap LF=L \ \text{and}\ L\cap F=K.$$ 

So $[LF:K]=[L:K][F:K]=nd$. Since $F/K$ is Galois, we have by Corollary \ref{isom K},
$$M\ \text{is isomorphic to}\ LF\ \text{over}\ K \iff \ M=L'F\ \text{where}\ L'\ \text{is isomorphic to}\ L\ \text{over}\ K. $$ 

Since such $L'$ is unique. Hence by Lemma \ref{L isom}, $s_K(LF)=s_K(L)$. Hence by Thm.\ \ref{Perlis} (2), we have
 \[ [LF:K] /r_K(LF)=[L:K]/r_K(L).\] Thus we get $r_K(LF)=r_K(L) [F:K]=rd$.
\end{proof}

\medskip

\begin{cor}
\hfill \medskip

{\it \noindent
The extension $\tilde{L}K'$ is the Galois closure of $LK'/K'$.
}
\end{cor}

\begin{proof}
By Lemma \ref{isomorphism}, the fields isomorphic to $LK'/K'$ are $L_iK'/K'$ where $L_i/K$ are distinct fields isomorphic to $L/K$. From remark \ref{tilde L}, $\tilde{L}= L_1 L_2\dots L_{s_K(L)}$. Hence the Galois Closure of $LK'/K'$ is $L_1 K' L_2 K' \dots L_{s_K(L)} K'=\tilde{L}K'$. 
\vspace{-0.5 cm}


\end{proof}

\begin{lem}
\label{Galois preserved}
{\it 
The degrees and Galois groups are preserved under base change. Hence the cluster sizes satisfy
 \[ r_K(L)= r_{K'}(LK').\]
 }
\end{lem}

\begin{proof}

By Lemma \ref{jacob}, since $K'/K$ is Galois, we have $$\tilde{L}\cap K'=K\iff \tilde{L}\cap LK'=L\ \text{and}\ L\cap K'=K.$$ Since, $L\cap K'=K$, we have $[L:K]=[LK':K']$. Also $[\tilde{L}:K]=[\tilde{L}K':K']$. Hence also $[\tilde{L}:L]=[\tilde{L}K':LK']$.\smallskip

   Let $G_1= {\rm Gal}(\tilde{L}K'/K')$. Let $H_1\subset G_1$ be subgroup with $LK'$ as the fixed field that is $H_1= {\rm Gal}(\tilde{L}K'/LK')$. By \cite[Thm. 2.6]{conrad2023galois}, since $\tilde{L}\cap K'=K$, we have $G_1$ is isomorphic to $G$ by restriction. And since $\tilde{L}\cap LK'=L$, we have $H_1$ is isomorphic to $H$ under same isomorphism.\smallskip

   Now the last assertion follows from Thm.\ \ref{Perlis} (3).
\end{proof}
\smallskip

\begin{rem}

Lem.\ \ref{Galois preserved} gives an alternate proof for Thm. \ref{basechange} by using the criterion in Thm. \ref{criterion}.  

\end{rem}


We conclude the following result for the strong cluster magnification of polynomials from Thm. \ref{basechange}.\smallskip

\begin{thm} 
{\it
   Suppose $g$ over $K$ is obtained by nontrivial strong cluster magnification from an $f$ over $K$ with magnification factor $d$. Let $K_g/K$ have a base change of $K$ by $K'$. Then $g$ over $K'$ is obtained by nontrivial strong cluster magnification from $f$ over $K'$ with the same magnification factor $d$.}
\end{thm}
\smallskip

Now we state a base change theorem for weak cluster magnification.

\begin{thm}

\label{basechangeweak}

{\it Suppose $M/K$ is obtained by weak cluster magnification from an $L/K$ with magnification factor $d$. Let $M/K$ have a base change of $K$ by $K'$. Then $MK'/K'$ is obtained by weak cluster magnification from $LK'/K'$ with the same magnification factor $d$.}
\end{thm}

\begin{proof}
Follows from last assertion in Lemma \ref{Galois preserved}. \end{proof}
\smallskip

\subsection{Base change and root capacity}\hfill

\begin{lem}
    \label{unique N}


{\it Collection of intermediate fields $M$ lying between $K'$ and $LK'$ is in
bijective correspondence with collection of intermediate fields $N$ lying between $K$ and $L$ by means of the correspondence $N \mapsto NK'$. For a given $M$, corresponding $N=\tilde{L}\cap M$.
}

\end{lem}

\begin{proof}
     From Lemma \ref{Galois preserved}, number of subgroups of ${\rm Gal}(\tilde{L}/K)$ containing ${\rm Gal}(\tilde{L}/L)=$ number of intermediate extensions for $L/K=$ number of intermediate extensions for $LK'/K'$. The other statement follows from Lemma \ref{jacob}.
\end{proof}
\smallskip

\begin{thm}\label{base change root capacity}
 {\it    Consider $M/L/K$. Let $M/K$ have a base change of $K$ by $K'$. Then $\rho_K(M,L)=\rho_{K'}(MK',LK')$.}
    \end{thm}

    \begin{proof}

    Let $\rho_K(M,L)=a_K \cdot r_K(L)$ and $\rho_{K'}(MK',LK')=a_{K'} \cdot r_{K'}(LK')$ where $a_K, a_{K'}$ are as in Lem.\ \ref{perlis rho}. From Lem.\ \ref{Galois preserved}, we have $r_K(L)=r_{K'}(LK')$.\smallskip

    By Prop.\ \ref{a} (1), $a_K$ is number of distinct fields inside $M\cap \tilde{L}$ isomorphic to $L$ over $K$ and $a_{K'}$ is number of distinct fields inside $MK'\cap \tilde{L}K'$ isomorphic to $LK'$ over $K'$. By Lem.\ \ref{unique N},  $K'\subset P\subset \tilde{L}K'\iff$ $P=L_1K'$ for a unique field $K\subset L_1 \subset  \tilde{L}$.\smallskip

Now for $L_1 \subset  \tilde{L}$, we claim $L_1\subset M\cap \tilde{L}\iff L_1K'\subset MK'\cap \tilde{L}K'$. Suppose 
$L_1K'\subset MK'\cap \tilde{L}K'$. Thus $L_1K'\cap \tilde{M} \subset (MK'\cap \tilde{M})\cap (\tilde{L}K'\cap \tilde{M})$. By Lemma \ref{jacob}, we have $L_1\subset M\cap \tilde{L}$. The other implication is clear. Hence
by Lem.\ \ref{isomorphism}, we are done.\end{proof}

\subsection{Base change and unique chains}\hfill

We study how the unique chains are influenced by base change.\smallskip

\begin{thm}

\label{base change unique chain}
  
\hfill    
 {\it    
    \begin{enumerate}
\smallskip        \item   $L\supsetneq N_1\supsetneq \dots \supsetneq N_k$ is unique descending chain for $L/K$ $\iff$ $LK'\supsetneq N_1K'\supsetneq \dots \supsetneq N_kK'$ is unique descending chain for $LK'/K'$ for $N_i\subset L$ for all $i$.

\smallskip        \item $K\subsetneq F_1\subsetneq \dots \subsetneq F_l$ is unique ascending chain for $L/K$ $\iff$ $K'\subsetneq F_1K'\subsetneq \dots \subsetneq F_lK'$ is unique ascending chain for $LK'/K'$ for $F_i\subset L$ for all $i$.
    \end{enumerate}
    }
\end{thm}

\begin{proof}

We use Lemma \ref{Galois preserved} here.\smallskip

Proof of (1): Since $\tilde{L}$ and $K'$ are linearly disjoint over $K$ and $N_i\subset L$ for all $i$, we have that $\tilde{N_i}$ and $K'$ are linearly disjoint over $K$ for all $i$. It is enough to show that the unique $N$ for $LK'/K'$ is $N_1K'$. \smallskip

Since ${\rm Gal}(\tilde{L}K'/K')\cong {\rm Gal}(\tilde{L}/K)=G$ through restriction and ${\rm Gal}(\tilde{L}K'/LK')\cong {\rm Gal}(\tilde{L}/L)=H$ under same map. Hence by identifying the groups, we have 
\[ \tilde{L}^{N_G(H)}=N_1\iff (\tilde{L}K')^{N_G(H)}=N_1K'.\]

Proof of (2): Let $L_i/F_i$ be the Galois closure of $L/F_i$ for all $i\geq 1$. Since $\tilde{L}$ and $K'$ are linearly disjoint over $K$, we have that $L_i$ and $K'$ are linearly disjoint over $K$. Hence by Lemma \ref{jacob}, $L_i$ and $F_i K'$ are linearly disjoint over $F_i$ for all $i$. \smallskip

It is enough to show that the unique $F$ for $LK'/K'$ is $F_1 K'$. Similar to proof of part (1), we have \\
$\tilde{L}^{H^G}=F_1\iff (\tilde{L}K')^{H^G}=F_1K'$.\end{proof}

\smallskip

\subsection{Strong cluster magnification and unique chains}
\hfill

\begin{thm}\label{strong unique descending}
 {\it   Let $M/K$ be obtained by strong cluster magnification from $L/K$ with $r_K(L)\neq 1$. Then we have that $M \supsetneq N_1\supsetneq \dots \supsetneq N_k$ is the unique descending chain for $M/K$ $\iff$ $L\supsetneq N_1\supsetneq \dots \supsetneq N_k$ is the unique descending chain for $L/K$.}
\end{thm}

\begin{proof}
    It is enough to show that $\tilde{M}^{N_{G'}(H')}=\tilde{L}^{N_{G}(H)}$ where $G'={\rm Gal}(\tilde{M}/K)$ and 
    $H'={\rm Gal}(\tilde{M}/M)$. After identifying the groups in Prop.\ \ref{prop}, we have $G'=G\times R$ and $H'=H\times 1$. Hence $N_{G'}(H')=N_{G}(H)\times R$. Now $\tilde{L}=\tilde{M}^{1\times R}$. Hence $\tilde{M}^{N_{G}(H)\times R}\subseteq\tilde{L}$ and ${\rm Gal}(\tilde{L}/\tilde{M}^{N_{G'}(H')})=N_G(H)$. Thus we are done. \smallskip
    
    We could also see the last fact in the following way. By Lemma \ref{jacob}, $\tilde{L}\cap F =K \iff N\cap F=K \ \text{and}\ \tilde{L}\cap NF = N$. We also have $(\tilde{L})(NF)=\tilde{L}F=\tilde{M}$. Hence $Gal(NF/N)\cong Gal(F/K)$ and $Gal(\tilde{M}/N)=Gal(\tilde{L}/N)\times Gal(NF/N)=N_G(H)\times R$.
\end{proof}

\begin{rem}
    If $r_K(L)=1$ in above theorem. Then the unique descending chain for $M/K$ is $M\supsetneq L$.
\end{rem}

\begin{thm}\label{strong unique ascending}
 {\it   Let $M/K$ be obtained by strong cluster magnification from $L/K$ through $F/K$ as in Def \ref{SCM} with $t_K(L)\neq 1$. Then we have
    
    \begin{enumerate}
        \item $F'$ is unique intermediate field for $M/K$ as in Thm. \ref{unique ascending chain} $\iff$ $F'=F_1F$  where $F_1$ is unique intermediate field for $L/K$.

        \item $K\subsetneq F_1\subsetneq \dots \subsetneq F_k$ is the unique ascending chain for $L/K$ $\iff$ $K\subsetneq F_1F\subsetneq \dots \subsetneq F_kF$ is the unique ascending chain for $M/K$ for $F_i\subset L$ for all $i$.
    \end{enumerate}
    }

\end{thm}

\begin{proof}
(1) Let $G'={\rm Gal}(\tilde{M}/K)$ and $H'={\rm Gal}(\tilde{M}/M)$. After identifying the groups in Prop.\ \ref{prop}, we have $G'=G\times R$ and $H'=H\times 1$. Hence ${H'}^{G'}=H^G\times 1$. Now $F=\tilde{M}^{G\times 1}$. Let $F_1=\tilde{L}^{H^G}$. Hence $F_1=\tilde{M}^{H^G\times R}$. Since $(G\times 1)\cap (H^G\times R)=(H^G\times 1)$, we get $\tilde{M}^{H^G\times 1}=F_1F$. \medskip
    
   (2) Since $\tilde{L}=\tilde{M}^{1\times R}$. Thus $\tilde{L}\cap F_1F=F_1$. Since $M=LF$, by Thm.\ \ref{base change unique chain} (2), we have \\
    $F_1\subsetneq F_2\subsetneq \dots \subsetneq F_k$ is unique ascending chain for $L/F_1$ $\iff$ $F_1 F\subsetneq F_2 F\subsetneq \dots \subsetneq F_k F$ is unique ascending chain for $M/F_1F$ for $F_i\subset L$ for all $i\geq 2$.
\end{proof}

\begin{rem}
    If $t_K(L)=1$ in the above theorem. Then the unique ascending chain for $M/K$ is $K\subsetneq F$.
\end{rem}


\section{Properties of Ascending Index} \label{Properties of tKL} Recall that the ascending index $t_K(L)$ of $L/K$ was defined to be the degree $[F : K]$ in Thm. \ \ref{unique ascending chain}.\smallskip


\begin{prop}\label{prop-tKL}
\hfill \medskip

{\it \noindent
 If $M/L/K$ are extensions, then $t_K(L)|t_K(M)$.
    }
\end{prop}

\begin{proof}
    Let $F/K$ and $F'/K$ be the unique intermediate extensions as in Thm.\ \ref{unique ascending chain} for $L/K$ and $M/K$ respectively. Since $F'/K$ is Galois with maximum possible degree contained in $M$, we conclude that $F\subset F'$. Thus $[F:K] \ |\ [F':K]$. 
\end{proof}

\begin{rem}
    The above proposition tells that the analogue of the notion of weak magnification always holds for ascending index.
\end{rem}

\smallskip

From Lemma \ref{Galois preserved}, we have the following base change result for ascending index.\smallskip

\begin{prop}\label{basechange-tKL}
\hfill \medskip

{\it \noindent
    Let $L/K$ be a finite extension and let $K'/K$ be a finite Galois extension such that $\tilde{L}$ and $K'$ are linearly disjoint over $K$. Then \[ t_K(L)= t_{K'}(LK').\]
    }
\end{prop}
\smallskip

By proof of Thm.\ \ref{strong unique ascending} (1) and Thm.\ \ref{unique ascending chain} (2), we have the following analogue of the Cluster Magnification Theorem Thm.\ \ref{Vanchinathan}.
\smallskip

\begin{thm}
\label{t strong}
    (Ascending Index Magnification Theorem) 
    {\it Let $M/K$ be obtained by strong cluster magnification with magnification factor $d$. Then 
     $$t_K(M)=d \ t_K(L)\ \text{and}\ u_K(M) = u_K(L).$$}
\end{thm}

\smallskip

The following theorem is an analogue of Thm.\ \ref{n,r}.\smallskip

\begin{thm}
\label{t,r}
   (Inverse Ascending Index Problem for Number Fields)
   {\it  Let $K$ be a number field. Let $n>2$ and $t|n$. Then there exists an extension $L/K$ of degree $n$ with ascending index $t_K(L)=t$. }
\end{thm}

\begin{proof}

Suppose $t=1$. By Lemmas \ref{sn} and \ref{perm} and Thm.\ \ref{nPk} and Thm. \ref{unique chain Lk} we have $L=L_1$ which satisfies $t_K(L)=1$.  \smallskip

Now suppose $t=n$. By Thm.\ \ref{n,r} for $r=n$, there exists an $L/K$ of degree $n$ with $r_K(L)=n$. For that $L/K$ we have $t_K(L)=n$.\smallskip

Now suppose $1<t<n$. Hence $1< n/t < n$. By Thm.
 \ref{n,r} for $r=n/t$ and Thm.\ \ref{Perlis unique}, there exists an $L/K$ 
 of degree $n$ with $r_K(L)=n/t$ and $t_K(L)=t$. \end{proof}

\begin{rem}
    Since we have Thm.\ \ref{t strong}, we could have approached Thm.\ \ref{t,r} in a similar way as 
    Thm.\ \ref{n,r} was approached in \cite{krithika2023root} using Thm. \ref{Vanchinathan}. 
    This approach would have left us with some cases not covered as in Thm.\ 2 in \cite{krithika2023root}.
\end{rem}

The following theorem is an analogue of Thm.\ 2 in \cite{krithika2023root}.

\begin{thm}
    \hfill
{\it 
\begin{enumerate}
    \item Let $K$ be a number field. For any integers $u\geq 3, t\geq 1 $ there exists $L/K$ of degree $ut$ and $t_K(L)=t$.

    \item Let $K=\mathbb{Q}$. For $t$ an even number, there exists $L/K$ of degree $2t$ and $t_K(L)=t$.
    
\end{enumerate}  
}
\end{thm}

\begin{proof}

(1) For $n=u$, we have $L_1/K$ with degree $u$ and satisfying $t_K(L_1)=1$ as in first case of Thm.\ \ref{t,r}. By Lemma 2 in \cite{krithika2023root}, there exists $F/K$ Galois of degree $t$ such that $\tilde{L_1}$ and $F$ are linearly disjoint over $K$. Hence by Thm.\ \ref{t strong}, $L=L_1F$ has degree $ut$ over $K$ and $t_K(L)=t$.\smallskip

(2) Consider $P=\mathbb{Q}(\sqrt[4]{2})$ which has degree $4$ over $K=\mathbb{Q}$ and $t_K(P)=2$. For $t$ an even number, by Lemma 2 in \cite{krithika2023root}, there exists $F/K$ Galois of degree $t/2$ such that $\tilde{P}$ and $F$ are linearly disjoint over $K$. Hence by Thm.\ \ref{t strong}, $L=PF$ has degree $2t$ over $K$ and $t_K(L)=t$. \end{proof}

\medskip

\noindent {\it Acknowledgements:} Both the Authors would like to thank Prof. P Vanchinathan for introducing his beautiful paper \cite{krithika2023root} to us by email.  Shubham Jaiswal would like to acknowledge support of IISER Pune Institute Scholarship during this work.


\end{document}